\documentclass[review,hidelinks,onefignum,onetabnum]{siamart251216}
\usepackage{silence}
\usepackage{lipsum}
\usepackage{amsfonts}
\usepackage{graphicx}
\usepackage{epstopdf}
\usepackage{algorithmic}
\ifpdf
  \DeclareGraphicsExtensions{.eps,.pdf,.png,.jpg}
\else
  \DeclareGraphicsExtensions{.eps}
\fi

\newsiamremark{remark}{Remark}
\newsiamremark{hypothesis}{Hypothesis}
\newsiamthm{assumption}{Assumption}

\crefname{hypothesis}{Hypothesis}{Hypotheses}
\newsiamthm{claim}{Claim}
\newsiamremark{fact}{Fact}
\crefname{fact}{Fact}{Facts}

\headers{Multiscale Spectral distribution}{Eric T. Chung, Lijian Jiang and Yajun Wang}

\title{Topological spectral gap, multiscale Weyl's law, and homogenization in high-contrast PDEs
\thanks{Submitted to the editors DATE.
\funding{Eric T. Chung's work is partially supported by the Hong Kong RGC General Research Fund (Project Nos. 14304525 and 14305624). Lijian Jiang acknowledges the support of the NSFC (Project No. 12271408).}}}

\author{
  Eric T. Chung\thanks{Department of Mathematics, The Chinese University of Hong Kong, Hong Kong SAR
  (\email{eric.t.chung@cuhk.edu.hk}).}
\and Lijian Jiang \thanks{School of Mathematical Sciences, Tongji University, Shanghai, P.R. China
  (\email{ljjiang@tongji.edu.cn}).}
\and Yajun Wang \thanks{Department of Mathematics, The Chinese University of Hong Kong, Hong Kong SAR
  (\email{yajunwang@cuhk.edu.hk}).}
  }

\usepackage{amsopn}

\usepackage{amsmath}
\usepackage{amssymb}
\usepackage{microtype}

\usepackage{subcaption}
\usepackage{graphicx}

\usepackage{booktabs}
\usepackage{multirow}
\usepackage{tabularx} 

\usepackage{verbatim}
\usepackage{microtype}

\ifpdf
\hypersetup{
  pdftitle={Topological spectral gap, multiscale Weyl's law, and homogenization in high-contrast PDEs},
  pdfauthor={Eric T. Chung, Lijian Jiang and Yajun Wang}
}
\fi

\begin{document}

\maketitle

\begin{abstract}
This paper introduces a unified abstract variational and topological framework to characterize the spectral gap and eigenvalue distribution in high-contrast multiscale partial differential equations (PDEs). We rigorously prove that the exact location of the spectral gap is universally determined by the dimension of the local null space associated with the high-contrast inclusions. For systems with infinite-dimensional kernels, this location is strictly determined by the topological Betti numbers. Furthermore, we establish a multiscale Weyl's law via a spectral decoupling theorem, demonstrating that as the contrast approaches infinity, the multiscale spectrum bifurcates into two independent components: the Dirichlet spectrum of the background matrix and the internal Neumann spectrum of the inclusions. Using spectral homogenization theory, we also show that in the limit of vanishing periodicity, this expanding topological eigenspace asymptotically spans the entire spectral space of the macroscopic homogenized operator. These theoretical results are comprehensively verified through numerical experiments on diffusion, elasticity, fourth-order plate, Maxwell, and grad-div operators.
\end{abstract}



\section{Introduction}
\label{sec:intro}
In the modeling of complex multiscale phenomena, one naturally encounters partial differential equations (PDEs) with highly oscillatory and high-contrast coefficients.
These equations govern a wide range of physical systems, including scalar diffusion, linear elasticity in advanced composites, fourth-order plate mechanics, time-harmonic Maxwell's electromagnetics in metamaterials, and grad-div operator in highly heterogeneous porous media.
Accurate numerical simulation of these systems presents significant computational challenges, as the complex microscopic geometric structures and extreme material contrasts must be properly resolved. 

To reduce this computational cost, a variety of advanced multiscale model reduction methods and robust domain decomposition preconditioners have been developed in recent decades.
Classical multiscale methods include homogenization theory \cite{cioranescu1999introduction}, the Multiscale Finite Element Method (MsFEM) \cite{efendiev2009multiscale}, the Variational Multiscale Method (VMS) \cite{hughes1998variational}, and the Heterogeneous Multiscale Method (HMM) \cite{abdulle2012heterogeneous}.
However, to achieve robust approximation in highly heterogeneous and high-contrast media, the core mechanism of many methods increasingly relies on solving local spectral problems.
Representative multiscale model reduction methods driven by this spectral technique include the Generalized Multiscale Finite Element Method (GMsFEM) \cite{efendiev2013generalized,chung2014generalized,chung2023multiscale}, Localized Orthogonal Decomposition (LOD) \cite{maalqvist2014localization,henning2014localized,altmann2021numerical}, Multiscale Generalized Finite Element Method (MS-GFEM) \cite{babuvska2020multiscale,ma2022novel,ma2025unified}, Constraint Energy Minimizing Generalized Multiscale Finite Element Method (CEM-GMsFEM) \cite{chung2018constraint}, and Localized Subspace Iteration (LSI) \cite{guan2025localized}.
Concurrently, to accelerate iterative solvers, robust domain decomposition preconditioners that rely on local generalized eigenvalue problems to construct highly effective coarse spaces have gained significant attention. 
Notable contributions in this area include the pioneering spectral coarse spaces for overlapping Schwarz methods \cite{galvis2010domain, efendiev2012robust}, the GenEO (Generalized Eigenproblems in the Overlaps) framework \cite{spillane2014abstract, dolean2015introduction}, and adaptive spectral coarse spaces for non-overlapping BDDC and FETI-DP algorithms \cite{kim2015bddc, kim2017bddc, chung2021learning, klawonn2015feti}.

However, the robustness and efficiency of these multiscale coarse spaces rely on capturing the dominant low-frequency physics, which requires a rigorous mathematical understanding of the local spectral behavior and the existence of a spectral gap.
Although the existence of the spectral gap has been rigorously established for scalar diffusion problems \cite{galvis2010domain}, this fundamental result was previously treated as a problem-specific phenomenon.
In this paper, we reveal that this spectral behavior is fundamentally governed by the dimension of the operator's local null space.
Motivated by this insight, we introduce a unified variational framework applicable to any high-contrast PDE characterized by a finite-dimensional kernel, extending the spectral theory to complex vector-valued fields (e.g., linear elasticity) and higher-order operators (e.g., biharmonic plates).

Unfortunately, this theoretical framework faces a profound analytical limitation when applied to systems exhibiting infinite-dimensional kernels, such as time-harmonic Maxwell's equations and grad-div operator.
In these systems, the standard dimensionality rules fail. Because the null space is infinite-dimensional, the physical discrete spectrum is hidden by an infinite number of zero eigenvalues.
To mathematically recover the discrete spectrum and extend our unified framework, we employ the theory of Hilbert complexes and algebraic topology \cite{Arnold2006, BottTu1982} as fundamental analytical tools. 
Specifically, by utilizing the De Rham exact sequence and the $L^2$-orthogonal generalized Hodge decomposition, the infinite-dimensional redundant exact forms can be quotiented out. 
This reveals that the effective local energy-free modes are intrinsically isomorphic to finite-dimensional harmonic spaces. 
Crucially, we prove that the dimensions of these harmonic spaces are no longer mere algebraic quantities, but are strictly determined by fundamental topological invariants—specifically, the Betti numbers ($\beta_k$), which precisely quantify the multidimensional holes within the high-contrast inclusions. 
Consequently, we establish the existence of a universal topological spectral gap, demonstrating that its exact location is governed purely by the macroscopic topology of the complex media.

Above the topological spectral gap, characterizing the high-frequency spectrum under the extreme high-contrast limit $\eta \rightarrow \infty$ constitutes another significant challenge. 
Traditional eigenvalue estimates fail to capture the energy localization and internal resonance effects that typically emerge within highly heterogeneous media. 
To resolve this, we establish a spectral decoupling theorem based on a novel contrast-weighted $L^2$-normalization technique. 
We identify a mass concentration dichotomy: as the contrast parameter approaches infinity, the physical energy undergoes a mathematical bifurcation. 
The multiscale spectrum splits into two independent, invariant limiting components: the Dirichlet spectrum of the background matrix and the internal Neumann spectrum of the inclusions. 
By partitioning these disparate wave behaviors, we restructure the eigenvalue counting formulations based on classical Weyl's law to establish a new multiscale Weyl's law. 
This law provides an explicit, quantitative linear superposition of background matrix propagation and isolated inclusion resonance, offering a complete characterization of the high-frequency spectral distribution.

Finally, we extend our analysis to periodic high-contrast PDEs to investigate the connection between the topological spectral gap and homogenization theory. By utilizing two-scale convergence and spectral homogenization theory, we demonstrate a robust band separation where the relative spectral gap ratio remains invariant with respect to the periodicity scale. Furthermore, we establish a spectral completeness theorem, rigorously proving that in the limit of vanishing periodicity, the expanding eigenspace of these topological low-frequency modes asymptotically exhausts the entire infinite-dimensional spectral space of the macroscopic homogenized operator.

The research significance and theoretical implications of this unified framework are threefold. First, it complements and advances the foundational theory of spectral-based multiscale reduction and preconditioning methods, providing crucial analytical tools for contrast-independent deductions, error analysis of GMsFEM, and the convergence rates of localized subspace iterations (LSI). Second, it offers an a priori mechanism to infer the exact location of the spectral gap based purely on the properties of the governing PDEs and material coefficients, thereby enabling a deterministic and optimal selection of coarse spaces before the actual computation. Third, by the De Rham complex and the multiscale Weyl's law, this work expands the multiscale spectral analysis toward complex topological systems and high-frequency regimes, offering a powerful homological-variational tool to analyze wave propagation, localized resonance, and band-gap structures in advanced metamaterials and highly heterogeneous porous media.

The remainder of the paper is organized as follows. \Cref{sec:spectral_gap,sec:d_dimensions} introduce a unified variational framework to characterize the spectral gap for general high-contrast PDEs in $d$-dimensional space, detailing local null space dimensions and coercivity bounds. \Cref{sec:derham} extends this framework to systems with infinite-dimensional kernels (e.g., Maxwell and grad-div operator) by utilizing De Rham complexes and Betti numbers to establish a universal topological spectral gap. \Cref{sec:spec_decoupling} investigates the high-frequency regime, proving a spectral decoupling theorem and formulating a multiscale Weyl's law. \Cref{sec:homogenization} connects these topological modes with spectral homogenization theory to demonstrate robust band separation and spectral completeness. Finally, \cref{sec:num_ex} presents comprehensive numerical validations across various physical models, followed by conclusions in \cref{sec:conclu}.

\section{Problem setting and spectral gap}
\label{sec:spectral_gap}
In this section, we first briefly review the well-established spectral behavior for the scalar diffusion problem. Motivated by these classical results, we then propose a unified abstract variational framework that characterizes the spectral gap for general physical operators, including linear elasticity and fourth-order biharmonic equations.

\subsection{The scalar diffusion problem}
\label{subsec:diffusion_baseline}
Let $\Omega \subset \mathbb{R}^d$ be a bounded domain. We assume the medium contains $M$ disconnected high-contrast inclusions, denoted by $D = \bigcup_{i=1}^M D_i$. The background matrix is defined as $\Omega_0 = \Omega \setminus \overline{D}$. 

The classical high-contrast eigenvalue problem for scalar diffusion seeks $(\lambda^\eta, u^\eta) \in \mathbb{R} \times H_0^1(\Omega)$ such that:
\begin{equation}
    \int_{\Omega} \kappa^\eta(x) \nabla u^\eta \cdot \nabla v \, dx = \lambda^\eta \int_{\Omega} \rho^\eta(x) u^\eta v \, dx, \quad \forall v \in H_0^1(\Omega),
\end{equation}
where the physical coefficient (e.g., conductivity) $\kappa^\eta(x)$ and the density $\rho^\eta(x)$ exhibit a high-contrast:
\begin{equation}
    \kappa^\eta(x), \rho^\eta(x) \sim 
    \begin{cases} 
      \eta \gg 1 & \text{if } x \in D, \\
      1 & \text{if } x \in \Omega_0.
    \end{cases}
\end{equation}

As demonstrated by Yalchin and Galvis (2010) \cite{galvis2010domain}, the asymptotic distribution of the eigenvalues for this specific operator is well-understood. Because the null space of the gradient operator inside each inclusion consists solely of constant functions, each inclusion contributes exactly one localized low-energy mode. Consequently, it has been rigorously proven that:
\begin{equation}
    \lambda_n^\eta = \mathcal{O}(\eta^{-1}) \to 0 \quad \text{for } n = 1, \dots, M,
\end{equation}
while the $(M+1)$-th eigenvalue is uniformly bounded from below by a constant $C$ independent of the contrast $\eta$:
\begin{equation}
    \lambda_{M+1}^\eta \ge C > 0 \quad \text{as } \eta \to \infty,
\end{equation}
This creates a sharp spectral gap at $n = M+1 $.

\subsection{A unified variational framework for general operators}
\label{subsec:unified_framework}
Unlike the scalar diffusion problem, where the spectral gap occurs at $n = M+1$, vector-valued equations and higher-order operators exhibit more complex null spaces. To accommodate these complex null spaces, we now introduce a unified abstract variational framework.

Let $V(\Omega)$ be an appropriate Hilbert space. We consider the abstract generalized spectral problem:
\begin{equation}
\label{eq:general_spectral}
    a^\eta(u, v) = \lambda^\eta m^\eta(u, v), \quad \forall v \in V(\Omega),
\end{equation}
where 
\begin{equation}
\begin{aligned}
    & a^\eta(u, v) = \int_{\Omega} \kappa^{\eta} \langle \mathcal{L} u,\mathcal{L} v \rangle \, dx = \int_{\Omega_0} \kappa^{\eta} \langle \mathcal{L} u,\mathcal{L} v \rangle \, dx + \eta \sum_{i=1}^M \int_{D_i} \frac{\kappa^{\eta}}{\eta} \langle \mathcal{L} u,\mathcal{L} v \rangle, \\
    & m^\eta(u, v) = \int_{\Omega} \rho^{\eta} \langle u, v \rangle \, dx = \int_{\Omega_0} \rho^{\eta} \langle u, v \rangle \, dx + \eta \sum_{i=1}^M \int_{D_i} \frac{\rho^{\eta}}{\eta} \langle u, v \rangle .
    \end{aligned}
\end{equation}
$\mathcal{L}$ is the differential operator depending on the specific problem. We decompose the bilinear forms into background matrix and inclusion components:
\begin{equation}
    a^\eta(u, v) = a_0(u, v) + \eta a_1(u, v), \quad m^\eta(u, v) = m_0(u, v) + \eta  m_1(u, v),
\end{equation}
where $a_0, a_1, m_0, m_1$ are independent of $\eta$.

\textbf{Definition of the local null space dimension ($R$):}
The core of our theory relies on the algebraic dimension of the energy-free modes inside inclusions. We define the local null space $\mathcal{N}(D_i) = \{ u \in V(D_i) \mid \mathcal{L}u = 0 \}$ and denote its dimension by $R = \dim(\mathcal{N}(D_i))$. Depending on physics, $R$ varies significantly:
\begin{enumerate}
    \item \textbf{Scalar diffusion:} $\mathcal{L} = \nabla u $. The kernel consists of constant functions. Thus, \textbf{$R = 1$}.
    \item \textbf{Linear elasticity:} $\mathcal{L} = \varepsilon(u)$. The kernel consists of rigid body modes. Thus, \textbf{$R = 3$} (in 2D).
    \item \textbf{Fourth-order plate:} $\mathcal{L} = \nabla^2 u$. The kernel consists of linear polynomials $\text{span}\{1, x, y\}$. Thus, \textbf{$R = 3$} (in 2D).
\end{enumerate}

Based on this local null space, we impose the following fundamental assumption to guarantee the existence of a spectral gap.
\begin{assumption}[Abstract coercivity]
\label{assum_coer}
There exists a positive coercivity constant $C_{\mathrm{coer}}$, independent of the contrast parameter $\eta$, such that for any inclusion $D_i$ and any test function $u \in V(D_i)$ that is $m_1$-orthogonal to the local null space $\mathcal{N}(D_i)$, the following lower bound holds:
\begin{equation}
    a_1(u,u) \ge C_{\mathrm{coer}} m_1(u,u).
\end{equation}
As will be explicitly detailed in \cref{sec:d_dimensions}, this abstract condition is naturally fulfilled by standard functional analysis theorems, including the Poincaré inequality, the Second Korn's inequality, and the higher-order Bramble-Hilbert lemma.
\end{assumption}

\begin{theorem}[Universal spectral gap]
\label{thm:unified_spectral_gap}
Let $N = R \times M$ be the global null space dimension generated by the $M$ inclusions. As $\eta \to \infty$, the first $N$ eigenvalues collapse to zero: $\lambda_n^\eta \le \mathcal{O}(\eta^{-1})$ for $n = 1, \dots, N$. Furthermore, there exists a constant $C_{\rm gap} > 0$, independent of $\eta$, such that the spectral gap strictly emerges at the $(N+1)$-th eigenvalue: $\lambda_{N+1}^\eta \ge C_{\rm gap}$.
\end{theorem}

\begin{proof}
The proof relies on the Courant-Fischer Min-Max and Max-Min principles \cite{horn2012matrix}. The generalized Rayleigh quotient is given by
\begin{equation}
     R^\eta(u) = \frac{a_0(u, u) + \eta a_1(u, u)}{m_0(u, u) + \eta m_1(u, u)}.
\end{equation}

\textbf{Step 1: upper bound for $n \le R \times M$.} 
Let $\tilde{\mathcal{N}}(D_i)$ be the zero-extension of the local null space $\mathcal{N}(D_i)$ to the global domain $\Omega$. We define the $N$-dimensional global null space $S_N = \bigoplus_{i=1}^M \tilde{\mathcal{N}}(D_i)$. For any $u \in S_N$, $\mathcal{L}u = 0$ inside all inclusions, which yields $a_1(u,u) = 0$. By the Min-Max principle, for any $n \le N$, we can choose $S_n \subseteq S_N$ to obtain
\begin{equation}
    \lambda_n^\eta \le \max_{u \in S_N, u \ne 0} R^\eta(u) = \max_{u \in S_N, u \ne 0} \frac{a_0(u,u)}{m_0(u,u) + \eta m_1(u,u)} \le \frac{C}{\eta} = \mathcal{O}(\eta^{-1}).
\end{equation}

\textbf{Step 2: lower bound for $n = R \times M + 1$.} 
To establish the strict lower bound, we utilize the Max-Min formulation:
\begin{equation}
    \lambda_{N+1}^\eta = \max_{\dim(V_N)=N} \min_{u \perp V_N, u \ne 0} R^\eta(u).
\end{equation}
By choosing the specific $N$-dimensional subspace $V_N = S_N$, any valid test function $u^* \perp S_N$ naturally satisfies the local orthogonality condition $u^*|_{D_i} \perp \mathcal{N}(D_i)$ for all $i=1,\dots,M$. According to Assumption \ref{assum_coer}, the abstract coercivity holds:
\begin{equation}
    a_1(u^*,u^*) \ge C_{coer} m_1(u^*,u^*).
\end{equation}
Since the background energy is non-negative ($a_0(u^*,u^*) \ge 0$), the Rayleigh quotient for $u^*$ is rigorously bounded from below:
\begin{equation}
    R^\eta(u^*) \ge \frac{\eta a_1(u^*,u^*)}{m_0(u^*,u^*) + \eta m_1(u^*,u^*)} \ge \frac{\eta C_{coer} m_1(u^*,u^*)}{m_0(u^*,u^*) + \eta m_1(u^*,u^*)}.
\end{equation}
Dividing the numerator and denominator by $\eta m_1(u^*,u^*)$ and taking the asymptotic limit as $\eta \rightarrow \infty$, we obtain
\begin{equation}
    \lim_{\eta \rightarrow \infty} R^\eta(u^*) \ge \lim_{\eta \rightarrow \infty} \frac{C_{coer}}{\frac{m_0(u^*,u^*)}{\eta m_1(u^*,u^*)} + 1} = C_{coer}.
\end{equation}
Consequently, by the Max-Min principle, the eigenvalue satisfies $\lambda_{N+1}^\eta \ge C_{\rm gap} := C_{\rm coer} > 0$, which completes the proof.
\end{proof}

\section{Generalization and coercivity theorem in \texorpdfstring{$d$} --dimensions}
\label{sec:d_dimensions}

In \cref{sec:spectral_gap}, the abstract coercivity constant $C_{\rm coer}$ played a key role in ensuring the existence of the spectral gap independent of the contrast $\eta$. In this section, we specialize this abstract framework to general $d$-dimensional spatial domains $\Omega \subset \mathbb{R}^d$. 
We demonstrate how the local null space dimension $R_d$ varies with the spatial dimension $d$, and we identify the fundamental functional analysis theorems that provide the exact lower coercivity bounds for diffusion, elasticity and fourth-order problems.

\subsection{Scalar diffusion and the Poincaré inequality}
\label{sec:d_dimensions_diffu}
For the $d$-dimensional scalar diffusion (or thermal conduction) problem, the unknown is a scalar field $u \in H^1(\Omega)$.
\begin{itemize}
    \item \textbf{Differential operator:} The energy is governed by the gradient operator, $\mathcal{L}u = \nabla u$. The corresponding physical PDE in strong form is the classical high-contrast eigenvalue problem:
    \begin{equation}
        -\nabla \cdot (\kappa^\eta \nabla u) = \lambda^\eta \rho^\eta u \quad \text{in } \Omega.
    \end{equation}
    \item \textbf{Null space and dimension ($R_d$):} The local null space inside an inclusion $D_i$ is the space of functions with zero gradient, which consists of constant functions:
    \begin{equation}
        \mathcal{N}(D_i) = \{ u \in H^1(D_i) \mid \nabla u = \mathbf{0} \} = \text{span}\{1\}.
    \end{equation}
    Consequently, the dimension of the null space is invariant to the spatial dimension: $\mathbf{R_d = 1}$ for all $d \ge 1$.
    \item \textbf{Coercivity theorem:} The spectral gap lower bound $C_{\rm coer}$ is provided by the classical \textbf{Poincaré-Wirtinger inequality} \cite{evans2022partial}. For any $u \in H^1(D_i)$, if $u$ is orthogonal to the null space (i.e., its mean value $\bar{u}_{D_i} = 0$), then there exists a constant $C_P > 0$ depending only on the geometry of $D_i$ such that
    \begin{equation}
        \int_{D_i} |\nabla u|^2 \, dx \ge C_P \int_{D_i} |u|^2 \, dx.
    \end{equation}
\end{itemize}

\subsection{Linear elasticity and Korn's inequality}
For $d$-dimensional linear elasticity, the unknown is a vector displacement field $\mathbf{u} = (u_1, \dots, u_d)^T \in [H^1(\Omega)]^d$.
\begin{itemize}
    \item \textbf{Differential operator:} The elastic strain energy is formulated in terms of the symmetric gradient (strain tensor), $\mathcal{L}\mathbf{u} = \boldsymbol{\varepsilon}(\mathbf{u}) = \frac{1}{2}(\nabla \mathbf{u} + \nabla \mathbf{u}^T)$. We consider a simplified isotropic elastodynamic model where the material stiffness is characterized by a highly heterogeneous scalar coefficient $\kappa^\eta(\mathbf{x})$. The corresponding high-contrast eigenvalue problem is
    \begin{equation}
        -\nabla \cdot (\kappa^\eta \boldsymbol{\epsilon}(\mathbf{u})) = \lambda^\eta \rho^\eta \mathbf{u} \quad \text{in } \Omega.
    \end{equation}
    \item \textbf{Null space and dimension ($R_d$):} The condition $\boldsymbol{\varepsilon}(\mathbf{u}) = \mathbf{0}$ characterizes the space of Rigid Body Modes (RBM). In $\mathbb{R}^d$, a rigid body mode can be expressed as $\mathbf{u}(\mathbf{x}) = \mathbf{c} + \mathbf{W}\mathbf{x}$, where $\mathbf{c} \in \mathbb{R}^d$ is a translation vector and $\mathbf{W} \in \mathbb{R}^{d \times d}$ is a skew-symmetric matrix representing infinitesimal rotations.
    The number of translational degrees of freedom is $d$, and the number of independent rotational degrees of freedom (the dimension of $d \times d$ skew-symmetric matrices) is $d(d-1)/2$. Thus, the total null space dimension varies quadratically with $d$:
    \begin{equation}
        \mathbf{R_d} = d + \frac{d(d-1)}{2} = \mathbf{\frac{d(d+1)}{2}}.
    \end{equation}
    Specifically, $R_2 = 3$ (2D elasticity) and $R_3 = 6$ (3D elasticity).
    \item \textbf{Coercivity theorem:} The existence of the spectral gap relies on the \textbf{Second Korn's inequality} \cite{ciarlet1994three}. For any $\mathbf{u} \in [H^1(D_i)]^d$ orthogonal to the RBM space, there exists a constant $C_K > 0$ such that:
    \begin{equation}
        \int_{D_i} |\boldsymbol{\varepsilon}(\mathbf{u})|^2 \, dx \ge C_K \int_{D_i} |\mathbf{u}|^2 \, dx.
    \end{equation}
\end{itemize}

\subsection{Fourth-order problems and the Bramble-Hilbert lemma}
For fourth-order problems (e.g., Kirchhoff plates in 2D, or generalized phase-field/strain-gradient models in 3D), the energy functional involves the second-order derivatives. The unknown is $u \in H^2(\Omega)$.
\begin{itemize}
    \item \textbf{Differential operator:} The energy is governed by the Hessian matrix, $\mathcal{L}u = \nabla^2 u$. The corresponding strong form is the fourth-order biharmonic-type eigenvalue problem:
    \begin{equation}
        \nabla \cdot \nabla \cdot (\kappa^\eta \nabla^2 u) = \lambda^\eta \rho^\eta u \quad \text{in } \Omega.
    \end{equation}
    \item \textbf{Null space and dimension ($R_d$):} The condition $\nabla^2 u = \mathbf{0}$ implies that all second derivatives vanish. Thus, the local null space comprises all linear polynomials:
    \begin{equation}
        \mathcal{N}(D_i) = \mathcal{P}_1(D_i) = \text{span}\{1, x_1, x_2, \dots, x_d\}.
    \end{equation}
    Physically, these modes represent rigid planar/hyperplanar translations and tilts. The dimension of this polynomial space is linear with $d$:
    \begin{equation}
        \mathbf{R_d = d + 1}.
    \end{equation}
    Specifically, $R_2 = 3$ (2D plates) and $R_3 = 4$ (3D polyharmonic operators).
    \item \textbf{Coercivity theorem:} The abstract coercivity $C_{\rm coer}$ for this higher-order space is strictly provided by a special case of the \textbf{Bramble-Hilbert Lemma} (or the Higher-Order Poincaré Inequality) \cite{brenner2008mathematical}. For any $u \in H^2(D_i)$ that is $L^2$-orthogonal to $\mathcal{P}_1(D_i)$, there exists a constant $C_{BH} > 0$ such that:
    \begin{equation}
        \int_{D_i} |\nabla^2 u|^2 \, dx \ge C_{BH} \int_{D_i} |u|^2 \, dx.
    \end{equation}
\end{itemize}

\cref{tab:dimension_scaling} summarizes the universal $n = R_d \times M + 1$ rule. This table provides explicit guidelines for designing robust multiscale coarse spaces in any spatial dimension.
\begin{table}[htbp]
    \centering
    \caption{The null space dimension $R_d$ and the exact location of the spectral gap $n$ for $M$ inclusions in $d$-dimensional space.}
    \label{tab:dimension_scaling}
    \renewcommand{\arraystretch}{1.4}
    \begin{tabularx}{\textwidth}{lXXX} 
        \toprule
        \textbf{PDE system} & \textbf{Null space dim. ($R_d$)} & \textbf{Gap loc. (2D)} & \textbf{Gap loc. (3D)} \\
        \midrule
        Scalar diffusion & $1$ & $M + 1$ & $M + 1$ \\
        Linear elasticity & $\frac{d(d+1)}{2}$ & $3M + 1$ & $6M + 1$ \\
        Fourth-order plate & $d + 1$ & $3M + 1$ & $4M + 1$ \\
        \bottomrule
    \end{tabularx}
\end{table}

\section{De Rham complexes, Betti numbers, and the topological spectral gap}
\label{sec:derham}
The unified framework presented in \cref{sec:spectral_gap} and \cref{sec:d_dimensions} relies on the fact that the local null space dimension $R_d$ of the differential operator is finite, which classically corresponds to polynomial degrees or rigid body motions. 
However, when this framework is applied to systems like time-harmonic Maxwell's equation or grad-div operator, the paradigm fails. The classical dimensionality rule breaks down due to infinite-dimensional null spaces, requiring a shift from standard algebraic vector spaces to the theory of Hilbert complexes and algebraic topology \cite{Arnold2006, BottTu1982}.

\subsection{The De Rham complex and infinite-dimensional kernels}
Let $\Omega \subset \mathbb{R}^3$ be a bounded domain. The fundamental function spaces associated with these physical systems constitute the standard de Rham Hilbert complex:
\begin{equation}\label{eq:derham}
H^1(\Omega) \xrightarrow{\nabla} H(\text{curl}; \Omega) \xrightarrow{\nabla\times} H(\text{div}; \Omega) \xrightarrow{\nabla\cdot} L^2(\Omega).
\end{equation}

To abstract this, we consider a sequence of Hilbert spaces $\{V_k\}_{k=0}^3$ equipped with densely defined, closed linear operators $d_k : V_k \rightarrow V_{k+1}$ (where $d_0 = \nabla$, $d_1 = \nabla\times$, $d_2 = \nabla\cdot$). 
The fundamental homological property of this complex is the chain complex condition:
\begin{equation}
d_k \circ d_{k-1} = 0 \quad \implies \quad \text{Im}(d_{k-1}) \subseteq \text{Ker}(d_k).
\end{equation}

To connect this abstract complex to physical phenomena, we instantiate the corresponding high-contrast eigenvalue problems for $k=1$ and $k=2$:
\begin{itemize}
    \item \textbf{Time-harmonic Maxwell's equations ($k=1$):} The relevant energy space is $V_1 = H(\operatorname{curl}; \Omega)$, and the differential operator is $d_1 = \nabla \times$. The corresponding strong form is the curl-curl eigenvalue problem governing the electric field $\mathbf{u}$ in electromagnetic metamaterials:
    \begin{equation}
        \nabla \times (\kappa^\eta \nabla \times \mathbf{u}) = \lambda^\eta \rho^\eta \mathbf{u} \quad \text{in } \Omega,
    \end{equation}
    where $\kappa^\eta$ represents the highly heterogeneous inverse magnetic permeability and $\rho^\eta$ denotes the electric permittivity.

    \item \textbf{Grad-div operator ($k=2$):} The relevant energy space is $V_2 = H(\operatorname{div}; \Omega)$, and the differential operator is $d_2 = \nabla \cdot$. The corresponding strong form is 
    \begin{equation}
        -\nabla (\kappa^\eta \nabla \cdot \mathbf{u}) = \lambda^\eta \rho^\eta \mathbf{u} \quad \text{in } \Omega.
    \end{equation}
\end{itemize}

For a high-contrast inclusion $D_i$, the local bilinear form measures the energy associated with the operator $d_k$. The local null space is precisely $\text{Ker}(d_k|_{D_i})$.
Due to the chain complex condition, this kernel contains the entire image of the preceding operator, $\text{Im}(d_{k-1}|_{D_i})$. Consequently, for $k \ge 1$, the null space is strictly infinite-dimensional. 
If the generalized eigenvalue problem \cref{eq:general_spectral} is solved directly in $V_k$, the discrete spectrum will be masked by an essential spectrum of infinitely many zero eigenvalues, rendering the physical spectral gap unobservable.

\subsection{Hodge decomposition and the cohomology group}
To recover discrete spectrum and extract the meaningful low-frequency physics, we must quotient out the infinite-dimensional redundant spaces. 
For the local domain $D_i$, we define the $k$-th continuous cohomology group $H^k(D_i)$ as the quotient space of the closed forms $Z^k = \text{Ker}(d_k|_{D_i})$ by the exact forms $B^k = \text{Im}(d_{k-1}|_{D_i})$ \cite{Hatcher2002}:
\begin{equation}
H^k(D_i) := \frac{Z^k}{B^k} = \frac{\text{Ker}(d_k|_{D_i})}{\text{Im}(d_{k-1}|_{D_i})}.
\end{equation}
To realize this abstract quotient space in a functional analysis setting, we employ the Hodge decomposition \cite{Arnold2006}. 
The space $V_k(D_i)$ admits the direct sum decomposition:
\begin{equation}
V_k(D_i) = \text{Im}(d_{k-1}) \oplus \mathcal{H}_k(D_i) \oplus \text{Im}(d_k^*),
\end{equation}
where $d_k^*$ denotes the formal adjoint of $d_k$, and $\mathcal{H}_k(D_i)$ represents the space of harmonic forms. 
By Hodge theory, every equivalence class in the cohomology group $H^k(D_i)$ contains a unique harmonic representative. Thus, there is a canonical isomorphism $\mathcal{H}_k(D_i) \cong H^k(D_i)$. 
By restricting our global variational problem to the orthogonal complement $W \subset V_k(\Omega)$ defined by $u \perp \text{Im}(d_{k-1})$, we filter out the exact forms. Consequently, the effective local null space reduces exactly to the harmonic space $\mathcal{H}_k(D_i)$.

\subsection{Betti numbers and the dimension of harmonic spaces}
The fundamental bridge between our analytical framework and pure topology is provided by the celebrated De Rham Theory \cite{BottTu1982}. It reveals that the analytically defined cohomology group $H^k(D_i)$ completely captures the macroscopic shape of the domain. Consequently, the algebraic dimension of the harmonic space is no longer merely an abstract analytical quantity. Instead, it is strictly determined by a topological invariant—the $k$-th Betti number ($\beta_k$), which simply counts the multidimensional holes within $D_i$:
\begin{equation}
\dim(\mathcal{H}_k(D_i)) = \dim(H^k(D_i)) = \beta_k(D_i).
\end{equation}
This provides a rigorous quantification of the local energy-free modes based purely on the domain geometry \cite{edelsbrunner2010computational}:
\begin{itemize}
    \item \textbf{0-Forms ($k=0$, $d_0=\nabla$):} $\beta_0$ counts the connected components. For a solid inclusion, $\beta_0=1$, recovering $R=1$ in \cref{sec:d_dimensions_diffu}.
    \item \textbf{1-Forms ($k=1$, $d_1=\nabla\times$):} $\beta_1$ counts the unshrinkable 1D topological tunnels (genus). For a solid sphere, $\beta_1=0$; for a torus, $\beta_1=1$.
    \item \textbf{2-Forms ($k=2$, $d_2=\nabla\cdot$):} $\beta_2$ counts the 2D enclosed cavities. For a solid block, $\beta_2=0$; for a hollow shell, $\beta_2=1$.
\end{itemize}

\subsection{The universal topological spectral gap}
With the continuous spaces defined, we establish the topological equivalent of the universal spectral gap theorem. 
To explicitly factor out the infinite-dimensional essential spectrum associated with the exact forms, we define the gauge-fixed space $W$ as the $L^2$-orthogonal complement of the image of the preceding operator $d_{k-1}$:
\begin{equation}
    W := \{ u \in V_k(\Omega) \mid (u, d_{k-1} p)_{L^2(\Omega)} = 0, \quad \forall p \in V_{k-1}(\Omega) \}.
\end{equation}
For example, in the case of Maxwell equations ($k=1$), this restricts the test space to divergence-free vector fields.

\begin{theorem}[Topological spectral gap]
\label{thm:homo_spectral_gap}
Let $N_{\rm topo} = \sum_{i=1}^{M}\beta_k(D_i)$ be the aggregate $k$-th Betti number across all $M$ inclusions. As the contrast $\eta \rightarrow \infty$, the first $N_{\rm topo}$ eigenvalues in the gauge-fixed space collapse to zero: $\lambda_n^\eta \le \mathcal{O}(\eta^{-1})$ for $ n= 1, \dots, N_{\rm topo}$. Furthermore, the spectral gap rigorously emerges at the $(N_{\rm topo}+1)$-th eigenvalue, strictly bounded from below by a constant independent of $\eta$.
\end{theorem}
\begin{proof}
The proof of this theorem is similar to that of \cref{thm:unified_spectral_gap}.

\textbf{Step 1: Upper bound for $n \le N_{\rm topo}$.}
We construct an $N_{\rm topo}$-dimensional global test space $S_{N_{\rm topo}} = \bigoplus_{i=1}^M \tilde{\mathcal{H}}_k(D_i) \subset W$, where $\tilde{\mathcal{H}}_k(D_i)$ denotes the zero-extension of the local harmonic space to the global domain $\Omega$. Since any $u \in S_{N_{ \rm topo}}$ satisfies $d_k u = 0$ strictly inside every inclusion, the high-contrast penalty exactly vanishes: $\|d_k u\|_{L^2(D_i)}^2 = 0$. By the Min-Max principle, for any $n \le N_{\rm topo}$, evaluating the Rayleigh quotient on $S_{N_{ \rm topo}}$ yields:
\begin{equation*}
    \lambda_n^\eta \le \max_{u \in S_{N_{\rm topo}}, u \ne 0} R^\eta(u) = \max_{u \in S_{N_{ \rm topo}}, u \ne 0} \frac{\|d_k u\|_{L^2(\Omega_0)}^2}{\|u\|_{L^2(\Omega_0)}^2 + \eta \sum_{i=1}^M \|u\|_{L^2(D_i)}^2} = \mathcal{O}(\eta^{-1}).
\end{equation*}
$$$$

\textbf{Step 2: Lower bound for $n = N_{\rm topo} + 1$.}
To establish the strict location of the spectral gap, we utilize the Max-Min formulation:
\begin{equation*}
    \lambda_{N_{\rm topo}+1}^\eta = \max_{\dim(V)=N_{\rm topo}} \min_{u \perp V, u \ne 0} R^\eta(u).
\end{equation*}
By choosing the specific subspace $V = S_{N_{ \rm topo}}$, any valid test function $u^* \in W$ must satisfy the local orthogonality constraint $u^*|_{D_i} \perp \mathcal{H}_k(D_i)$ for all $i=1,\dots,M$. According to the closed range theorem for De Rham complexes \cite{Arnold2006,bru1992hilbert}, the operator $d_k$ is bounded from below on the orthogonal complement of its effective kernel. Thus, there exists a generalized coercivity constant $C_{\rm coer} > 0$ (such as the Maxwell-Poincaré or Friedrichs constant), dependent only on the geometry of $D_i$, such that
\begin{equation}
    \sum_{i=1}^M \|d_k u^*\|_{L^2(D_i)}^2 \ge C_{\rm coer} \sum_{i=1}^M \|u^*\|_{L^2(D_i)}^2.
\end{equation}
Since the background energy is non-negative ($\|d_k u^*\|_{L^2(\Omega_0)}^2 \ge 0$), substituting this lower bound into the Rayleigh quotient yields
\begin{equation}
    R^\eta(u^*) \ge \frac{\eta C_{\rm coer} \sum_{i=1}^M \|u^*\|_{L^2(D_i)}^2}{\|u^*\|_{L^2(\Omega_0)}^2 + \eta \sum_{i=1}^M \|u^*\|_{L^2(D_i)}^2}.
\end{equation}
Dividing the numerator and the denominator by $\eta \sum_{i=1}^M \|u^*\|_{L^2(D_i)}^2$ and taking the asymptotic limit as $\eta \rightarrow \infty$, we obtain
\begin{equation}
    \lim_{\eta \rightarrow \infty} R^\eta(u^*) \ge \lim_{\eta \rightarrow \infty} \frac{C_{\rm coer}}{\frac{\|u^*\|_{L^2(\Omega_0)}^2}{\eta \sum_{i=1}^M \|u^*\|_{L^2(D_i)}^2} + 1} = C_{\rm coer}.
\end{equation}
Consequently, by the Max-Min principle, $\lambda_{N_{\rm topo}+1}^\eta \ge C_{\rm coer} > 0$, mathematically confirming that the emergence of the spectral gap is fundamentally governed by the topological invariants $\beta_k(D_i)$.
\end{proof}

\cref{tab:topology} summarizes this universal topological rule. It provides explicit guidelines
for designing robust multiscale spaces.

\begin{table}[ht]
\centering
\caption{The unified topological spectral gap governed by the De Rham complex and Betti numbers ($\beta_k$) for high-contrast multiscale operators.}
\label{tab:topology}
    \renewcommand{\arraystretch}{1.4}
\begin{tabularx}{\textwidth}{lXXX} 
\hline
\textbf{PDE system} & \textbf{Function space} & \textbf{Operator} ($d_k$) & \textbf{Null space Dim.} ($R$) \\
\hline
Scalar diffusion & 0-forms $H^1(\Omega)$ & Gradient ($\nabla$) & $R = \beta_0$ (connected components) \\
Maxwell's equations & 1-Forms $H(\text{curl}; \Omega)$ & Curl ($\nabla\times$) & $R = \beta_1$ (topological tunnels) \\
Grad-div operator & 2-Forms $H(\text{div}; \Omega)$ & Divergence ($\nabla\cdot$) & $R = \beta_2$ (enclosed cavities) \\
\hline
\end{tabularx}
\end{table}

\begin{remark}
    The topological spectral gap established in \cref{thm:homo_spectral_gap} has profound implications for the approximation theory of high-contrast PDEs, particularly concerning the Kolmogorov $n$-width \cite{pinkus2012n}. In functional analysis, the Kolmogorov $n$-width measures the optimal worst-case approximation error of a solution manifold using any $n$-dimensional linear subspace. Our results rigorously imply that in high-contrast media, the $n$-width decays extremely slowly for $n \le N_{\rm topo}$ due to the localized $\mathcal{O}(\eta^{-1})$ energy modes. However, the emergence of the spectral gap guarantees a sharp decay in the Kolmogorov $n$-width exactly at the dimension $n = N_{\rm topo} + 1$. Consequently, the topological invariant $N_{\rm topo} = \sum_{i=1}^M \beta_k(D_i)$ not only quantifies the macroscopic geometry but also defines the minimal intrinsic dimension required for multiscale model reduction methods.
\end{remark}

\section{Spectral decoupling and the multiscale Weyl's law}
\label{sec:spec_decoupling}
\Cref{sec:spectral_gap,sec:d_dimensions,sec:derham} addressed the low-frequency modes governed by the null space dimension $N$ (or $N_{\rm topo}$) and the formation of the topological spectral gap. We now turn our attention to the asymptotic behavior of the spectrum beyond the gap. Specifically, we investigate the limit of the eigenvalue sequence $\{\lambda_{n}^{\eta}\}$ for a fixed index $n > N$ as the contrast $\eta \rightarrow \infty$. We will rigorously prove that the multiscale spectrum bifurcates into two independent invariant components, completely restructuring the classical Weyl's law \cite{weyl1912asymptotische,safarov1997asymptotic}. 

While the unified variational and topological frameworks established in \cref{sec:spectral_gap,sec:d_dimensions,sec:derham} are applicable to a wide class of high-contrast PDEs (including vector-valued fields and higher-order operators), for the sake of clarity and simplicity in notation, we will restrict our explicit asymptotic proofs in \cref{sec:spec_decoupling,sec:homogenization} to the scalar diffusion problem. It is important to emphasize that this is done without loss of generality. The fundamental mechanisms underlying the spectral decoupling, the multiscale Weyl's law, and the spectral completeness can be naturally extended to other complex physical systems following the abstract coercivity and Hodge decomposition theories discussed earlier.

\subsection{Variational setup}
We consider the weak formulation of the generalized spectral problem \cref{eq:general_spectral} with the contrast parameter $\eta$. To facilitate a rigorous asymptotic analysis, we impose a global $\eta$-weighted $L^2$-normalization on the eigenfunctions $u^\eta \in H_0^1(\Omega)$:
\begin{equation}\label{eq:normalization}
\|u^\eta\|_{L^2(\Omega_0)}^2 + \eta \sum_{i=1}^{M} \|u^\eta\|_{L^2(D_i)}^2 = 1.
\end{equation}

By the Courant-Fischer Min-Max principle, for any fixed index $n$, the eigenvalue $\lambda_n^\eta$ is uniformly bounded from above by choosing a $n$-dimensional test space $S_n \subset H_0^1(\Omega_0)$ extended by zero into $D$. For such test functions, the high-contrast penalty vanishes, yielding $\lambda_n^\eta \le C_n < \infty$, where $C_n$ is independent of $\eta$. 

Substituting $v = u^\eta$ into the variational equation and utilizing the normalization condition \eqref{eq:normalization}, the Rayleigh quotient immediately provides the following a priori gradient energy bounds:
\begin{equation}\label{eq:apriori_grad}
\|\nabla u^\eta\|_{L^2(\Omega_0)}^2 \le C_n, \quad \text{and} \quad \sum_{i=1}^{M} \|\nabla u^\eta\|_{L^2(D_i)}^2 \le \frac{C_n}{\eta}.
\end{equation}
Furthermore, normalization \eqref{eq:normalization} implies that the mass inside the inclusions decays as $\|u^\eta\|_{L^2(D)}^2 \le \eta^{-1}$. 

\subsection{The mass concentration dichotomy} The limit spectrum is entirely governed by the asymptotic distribution of the generalized mass. We define the background matrix mass concentration limit as:
\begin{equation}
\alpha = \lim_{\eta \to \infty} \|u^\eta\|_{L^2(\Omega_0)}^2 \in [0, 1].
\end{equation}
Depending on whether $\alpha > 0$ or $\alpha = 0$, the asymptotic behavior of the eigenfunctions exhibits a dichotomy, as established in the following theorem.

\begin{theorem}[Spectral decoupling] 
\label{theorem:spec_decou}
Let $\{\lambda_n^\eta\}_{n=1}^\infty$ be the eigenvalue sequence for the multiscale operator. For any fixed index $n > N$, there exists a subsequence such that $\lambda^\eta_n \to \lambda^* \in (0, \infty)$ as $\eta \to \infty$. Then, the limit $\lambda^*$ must belong to the union of two independent discrete spectra:
\begin{enumerate}
    \item The Dirichlet spectrum of the background matrix: $\sigma(\mathcal{L}_{\Omega_0}^{\rm Dir})$.
    \item The internal Neumann spectrum of the inclusions: $\sigma(\mathcal{L}_{D}^{\rm Neu}) \setminus \{0\}$.
\end{enumerate}
\end{theorem}

\begin{proof}
The proof proceeds by analyzing the two mutually exclusive cases of the mass concentration limit $\alpha$.

\textbf{Case 1: background matrix dominated mode ($\alpha > 0$).}
From the a priori estimates \eqref{eq:apriori_grad} and \eqref{eq:normalization}, the sequence $\{u^\eta\}$ is uniformly bounded in the Hilbert space $H_0^1(\Omega)$. 
By the Banach-Alaoglu theorem \cite{brezis2011functional}, we can extract a weakly convergent subsequence (still denoted by $u^\eta$). Furthermore, the Rellich-Kondrachov compact embedding theorem \cite{brezis2011functional} ensures that this sequence converges strongly in $L^2(\Omega)$. Thus, there exists a limit function $u^* \in H_0^1(\Omega)$ such that
\begin{equation}
    u^\eta \rightharpoonup u^* \quad \text{weakly in } H_0^1(\Omega), \quad \text{and} \quad u^\eta \to u^* \quad \text{strongly in } L^2(\Omega).
\end{equation}
Since the mass bound strictly yields $\|u^\eta\|_{L^2(D)}^2 \le \eta^{-1}$, the strong $L^2$-convergence implies $\|u^*\|_{L^2(D)} = 0$, meaning $u^*|_D \equiv 0$. Consequently, the trace theorem ensures $u^* \in H_0^1(\Omega_0)$. Furthermore, since $\|u^*\|_{L^2(\Omega_0)}^2 = \alpha > 0$, the limit $u^*$ is non-trivial.

For an arbitrary test function $v \in H_0^1(\Omega_0)$ extended by zero to $\Omega$, all integrals over $D$ vanish identically. Evaluating the variational equation \eqref{eq:general_spectral} on $v$ and passing to the limit $\eta \to \infty$ yields
\begin{equation}
\begin{aligned}
     \lim_{\eta \to \infty} \int_{\Omega_0} \nabla u^\eta \cdot \nabla v \, dx & = \lim_{\eta \to \infty} \lambda^\eta \int_{\Omega_0} u^\eta v \, dx \\
     & \implies \int_{\Omega_0} \nabla u^* \cdot \nabla v \, dx = \lambda^* \int_{\Omega_0} u^* v \, dx, \quad \forall v \in H_0^1(\Omega_0).
\end{aligned}
\end{equation}
This rigorously confirms that $\lambda^*$ belongs to the Dirichlet spectrum of the background matrix, i.e., $\lambda^* \in \sigma(\mathcal{L}_{\Omega_0}^{\rm Dir})$.

\textbf{Case 2: Internal resonance mode ($\alpha = 0$).}
In this case, the mass strictly concentrates within the inclusions: $\lim_{\eta \to \infty} \eta \|u^\eta\|_{L^2(D)}^2 = 1$. We introduce the rescaled sequence $\tilde{u}^\eta = \sqrt{\eta} u^\eta$, preserving the local mass limit $\|\tilde{u}^\eta\|_{L^2(D)}^2 \to 1$. The global variational equation is correspondingly rescaled as:
\begin{equation} \label{eq:rescaled_var}
\frac{1}{\eta} \int_{\Omega_0} \nabla \tilde{u}^\eta \cdot \nabla v \, dx + \int_D \nabla \tilde{u}^\eta \cdot \nabla v \, dx = \lambda^\eta \left( \frac{1}{\eta} \int_{\Omega_0} \tilde{u}^\eta v \, dx + \int_D \tilde{u}^\eta v \, dx \right), \quad \forall v \in H_0^1(\Omega).
\end{equation}
The rescaled gradient bound $\|\nabla \tilde{u}^\eta\|_{L^2(D)}^2 \le C_n$ ensures the existence of a local weak limit $\tilde{u}^\eta|_D \rightharpoonup \tilde{u}^*$ in $H^1(D)$. By compact embedding, $\tilde{u}^\eta \to \tilde{u}^*$ strongly in $L^2(D)$, which strictly preserves the local mass limit $\|\tilde{u}^*\|_{L^2(D)} = 1$ ($\tilde{u}^* \ne 0$).

To decouple the domains, for any unconstrained local test function $v \in H^1(D)$, we employ the Calder\'{o}n extension theorem \cite{evans2022partial} to construct a bounded global extension $E(v) \in H_0^1(\Omega)$ such that $E(v)|_D = v$ and $\|E(v)\|_{H^1(\Omega)} \le K \|v\|_{H^1(D)}$. Substituting $v = E(v)$ into \eqref{eq:rescaled_var}, the background gradient is bounded via the Cauchy-Schwarz inequality:
\begin{equation}
\begin{aligned}
    \left| \frac{1}{\eta} \int_{\Omega_0} \nabla \tilde{u}^\eta \cdot \nabla E(v) \, dx \right| & \le \frac{1}{\eta} \|\nabla \tilde{u}^\eta\|_{L^2(\Omega_0)} \|\nabla E(v)\|_{L^2(\Omega_0)} \\
    & \le \frac{\sqrt{C_n \eta}}{\eta} K \|v\|_{H^1(D)} = \mathcal{O}(\eta^{-1/2}).
\end{aligned}
\end{equation}
Similarly, the background mass term satisfies $\left| \frac{1}{\eta} \int_{\Omega_0} \tilde{u}^\eta E(v) \, dx \right| \le \mathcal{O}(\eta^{-1/2})$. Taking the asymptotic limit $\eta \to \infty$, the background contributions vanish, and the equation rigorously reduces to the local domain:
$$\int_D \nabla \tilde{u}^* \cdot \nabla v \, dx = \lambda^* \int_D \tilde{u}^* v \, dx, \quad \forall v \in H^1(D).$$
Applying integration by parts on the unconstrained space $H^1(D)$ yields the strong form:
$$\int_D (-\Delta \tilde{u}^* - \lambda^* \tilde{u}^*) v \, dx + \int_{\partial D} \frac{\partial \tilde{u}^*}{\partial n} v \, ds = 0, \quad \forall v \in H^1(D).$$
This naturally yields the homogeneous Neumann boundary condition $\frac{\partial \tilde{u}^*}{\partial n} = 0$ on $\partial D$. Thus, $\lambda^*$ is verified as an internal Neumann eigenvalue, completing the proof.
\end{proof}

\begin{remark}[Physical interpretation of the mass concentration limit $\alpha$]
In fact, the proof of case 2 is true for all cases where $\alpha < 1$.
The parameter $\alpha = \lim_{\eta\rightarrow\infty}||u^{\eta}||_{L^2(\Omega_0)}^2 $ serves as a mathematical metric for the asymptotic state of the mass distribution. The tripartite classification of $\alpha$ clarifies how the geometric and material features of the high-contrast medium govern the macro-micro mode decoupling:
\begin{itemize}
    \item \textbf{Macroscopic wave propagation ($\alpha = 1$):} when $\alpha = 1$, the generalized mass strictly concentrates within the continuous background matrix $\Omega_0$, while the mass inside the heavy inclusions vanishes ($||u^{\eta}||_{L^2(D)} \sim \mathcal{O}(\eta^{-1})$). In this case, the limiting eigenmodes behave as macroscopic waves that freely propagate through the background medium, governed by the Dirichlet spectrum of $\Omega_0$.
    \item \textbf{Localized microscopic resonance ($\alpha = 0$):} When $\alpha = 0$, the mass is trapped within the stiff inclusions $D$, forcing the background displacement field to scale as $||u^{\eta}||_{L^2(\Omega_0)} \sim \mathcal{O}(\eta^{-1/2})$. This corresponds to a localized resonance mode, where a microscopic Neumann eigenfunction inside $D$ acts as a boundary source that drives a decaying Helmholtz response in the background matrix $\Omega_0$. 
    \item \textbf{Accidental resonance degeneracy ($0 < \alpha < 1$):} The intermediate case $0 < \alpha < 1$ represents a rare physical regime where the generalized mass is distributed across both the macroscopic background matrix and the microscopic inclusions. From a mathematical perspective, this indicates an intersection or ``accidental degeneracy" between the two decoupled spectra, i.e., $\sigma(\mathcal{L}_{\Omega_0}^{\rm Dir}) \cap \sigma(\mathcal{L}_{D}^{\rm Neu}) \neq \emptyset$.
\end{itemize}
\end{remark}

\subsection{The multiscale Weyl's law}
In standard homogeneous media, the asymptotic distribution of eigenvalues is characterized by the classical Weyl's law \cite{weyl1912asymptotische,safarov1997asymptotic}. The eigenvalue counting function $N(\lambda) = \#\{n : \lambda_n \le \lambda\}$ scales continuously with the total volume $|\Omega|$:
\begin{equation}
N(\lambda) \sim \frac{|B_d| |\Omega|}{(2\pi)^d} \lambda^{d/2} \quad \text{as } \lambda \rightarrow \infty,
\end{equation}
where $|B_d|$ denotes the volume of the unit ball in $\mathbb{R}^d$.

\cref{theorem:spec_decou} restructures the spectrum distribution for multiscale high-contrast PDEs. Because the spectrum rigorously decouples, the total counting function $N_{total}(\lambda)$ for the eigenvalues behaves asymptotically as the linear superposition of two geometrically distinct physical mechanisms:
\begin{equation}\label{eq:generalized_weyl}
N_{\rm total}(\lambda) \approx N_{\Omega_0}^{\text{Dir}}(\lambda) + \sum_{i=1}^{M} N_{D_i}^{\text{Neu}}(\lambda).
\end{equation}
Equation \eqref{eq:generalized_weyl} represents the multiscale Weyl's law for multiscale high-contrast media. It elegantly reveals that high-frequency modes in such systems do not propagate as unified waves. Instead, they decompose into either macroscopic wave propagation through the complex background matrix ($\Omega_0$), or isolated local resonances trapped within the inclusions ($D_i$).

\section{Spectral gap and homogenization theory}
\label{sec:homogenization}
In the previous sections, we used algebraic topology methods to demonstrate the existence of $N$ low-frequency eigenmodes  corresponding to the local null space within high-contrast structures. 
In this section, we will turn our attention to periodic high-contrast PDEs.  We utilize two-scale convergence and spectral homogenization theory to rigorously prove that, in the limit of vanishing periodicity ($\varepsilon \to 0$), the eigenspace of these $N$ topological modes converges precisely and completely to the spectral space of the macroscopic effective operator.

\subsection{Spectral homogenization theory}

Consider a periodic high-contrast eigenvalue problem within a physical domain $\Omega \subset \mathbb{R}^d$. Let $\varepsilon$ denote the periodicity scale and $\eta$ denote the high-contrast parameter. The spectral problem for the weighted elliptic operator $\mathcal{A}^\varepsilon = -(\kappa^\varepsilon)^{-1}\nabla \cdot (\kappa^\varepsilon \nabla)$ can be formulated in the following variational form:
find $(\lambda^\varepsilon, u^\varepsilon) \in \mathbb{R}^+ \times H^1_0(\Omega)$ such that
\begin{equation} \label{eq:micro_var}
    \int_\Omega \kappa^\varepsilon(x) \nabla u^\varepsilon \cdot \nabla v \, dx = \lambda^\varepsilon \int_\Omega \rho^\varepsilon(x) u^\varepsilon v \, dx, \quad \forall v \in H^1_0(\Omega),
\end{equation}
where the coefficients $\kappa^\varepsilon(x) = \kappa(x/\varepsilon), \rho^\varepsilon(x) = \rho(x/\varepsilon)$ satisfy $\kappa, \rho \sim \mathcal{O}(1)$ in the background medium $\Omega_0$, and $\kappa,\rho \sim \mathcal{O}(\eta)$ within the inclusions $D$.

To extract the macroscopic asymptotic behavior of the system as $\epsilon \to 0$, we introduce the classical corrector functions $\chi_j(y)$ on the reference periodic unit cell $Y = [0,1]^d$, which satisfy the following cell problem:
\begin{equation} \label{eq:cell_problem}
    -\nabla_y \cdot \Big( \kappa(y) (\nabla_y \chi_j + \mathbf{e}_j) \Big) = 0, \quad \text{in } Y,
\end{equation}
where $\chi_j$ is $Y$-periodic and $\int_Y \chi_j dy = 0$.

Based on the solutions to the cell problem, we define the macroscopic effective tensor $\kappa^*$ and the effective density $\rho^*$ as follows:
\begin{equation}
\begin{aligned}
    \kappa_{ij}^* &= \frac{1}{|Y|} \int_Y \kappa(y) \left( \delta_{ij} + \frac{\partial \chi_i}{\partial y_j} \right) dy, \\
    \rho^* &= \frac{1}{|Y|} \int_Y \kappa(y) dy.
\end{aligned}
\end{equation}
Consequently, the homogenized eigenvalue problem defined on the macroscopic space $H^1_0(\Omega)$ is formulated as: find $(\lambda^0, U) \in \mathbb{R}^+ \times H^1_0(\Omega)$ such that
\begin{equation} \label{eq:macro_var}
    \int_\Omega \kappa^* \nabla U \cdot \nabla V \, dx = \lambda^0 \int_\Omega \rho^* U V \, dx, \quad \forall V \in H^1_0(\Omega)
\end{equation}
Then we have following spectral convergence lemma \cite{kesavan1979homogenization}. This lemma states the behavior of the spectrum in homogenization. 
\begin{lemma}  
\label{lemma:homo_spec_conver}
Let $\{(\lambda_n^\varepsilon, u_n^\varepsilon)\}_{n=1}^\infty \subset \mathbb{R}^+ \times H_0^1(\Omega)$ be the normalized eigenpairs of the multiscale operator $\mathcal{A}^\varepsilon$. As $\varepsilon \to 0$, the multiscale spectrum converges to the spectrum of the macroscopic homogenized operator $\mathcal{A}^0$. Specifically, for each $n \in \mathbb{N}$, the entire sequence of eigenvalues converges:
\begin{equation}
    \lim_{\varepsilon \to 0} \lambda_n^\varepsilon = \lambda_n^0.
\end{equation}
Moreover, there exists a subsequence (still denoted by $\varepsilon$) such that the corresponding eigenfunctions satisfy
\begin{equation}
    u_n^\varepsilon \rightharpoonup u_n^0 \quad \text{weakly in } H_0^1(\Omega), \quad \text{and} \quad u_n^\varepsilon \to u_n^0 \quad \text{strongly in } L^2(\Omega),
\end{equation}
where $\{(\lambda_n^0, u_n^0)\}_{n=1}^\infty$ are the eigenpairs of the homogenized spectral problem \cref{eq:macro_var}. The sequence of limits $\{u_n^0\}_{n=1}^\infty$ constitutes a complete orthonormal basis for $L^2(\Omega)$.

Furthermore, if $\lambda_n^0$ is a simple eigenvalue with algebraic multiplicity $m(\lambda_n^0) = 1$, then there exists a critical scale $\varepsilon_0(n) > 0$ such that for all $\varepsilon \le \varepsilon_0(n)$, the multiscale eigenvalue is also simple, i.e., $m(\lambda_n^\varepsilon) = 1$. In this case, for the normalized eigenfunction $u_n^0$, the entire sequence of $\{u_n^\varepsilon\}$ strictly converges without extracting subsequences:
\begin{equation}
    \lim_{\epsilon \to 0} \|u_n^\epsilon - u_n^0\|_{L^2(\Omega)} = 0.
\end{equation}
\end{lemma}

\subsection{Asymptotic spectral gap dependent on periodicity}
To connect microscopic topological modes with the macroscopic homogenized spectrum, we analyze the system under a fixed high contrast $\eta \gg 1$ while taking the limit of periodicity $\varepsilon \rightarrow 0$. Due to the $\varepsilon^{-2}$ scaling of the gradient operator on domains of size $\varepsilon$, we show that the spectrum exhibits a robust band separation. The relative spectral gap ratio remains invariant with respect to $\varepsilon$. This ensures that the topological separation in the spectrum is a fundamental property of the multiscale operator across all spatial scales.

\begin{theorem}
\label{thm:asy_spec_gap}
Let $\varepsilon \ll 1$ be the periodicity scale and $\eta \gg 1$ be the contrast parameter. Let $\lambda_{n}^{\epsilon}$ be the $n$-th eigenvalue of the multiscale system \eqref{eq:micro_var}. The spectrum exhibits a robust band separation characterized by the following $\varepsilon$- and $\eta$-dependent bounds:
\begin{itemize}
    \item Topological low-frequency band: For $n \le N(\varepsilon)$, $\lambda_{n}^{\varepsilon} \le \mathcal{O}(\varepsilon^{-2}\eta^{-1})$.
    \item High-frequency resonance band: For $n \ge N(\varepsilon)+1$, $\lambda_{n}^{\varepsilon} \ge \mathcal{O}(\varepsilon^{-2})$.
\end{itemize}
\end{theorem}

\begin{proof}
The proof relies on the Courant-Fischer Min-Max principle:
\begin{equation}
\lambda_{k}^{\varepsilon} = \min_{\dim(S_{k})=k} \max_{u \in S_{k}, u \ne 0} R^{\varepsilon}(u), \quad \text{where } R^{\varepsilon}(u) = \frac{\int_{\Omega_{0}^{\varepsilon}} |\nabla u|^{2} dx + \eta \int_{D^{\varepsilon}} |\nabla u|^{2} dx}{\int_{\Omega_{0}^{\varepsilon}} |u|^{2} dx + \eta \int_{D^{\varepsilon}} |u|^{2} dx}.
\end{equation}

\textbf{Step 1: upper bound for $n \le N(\varepsilon)$.} 
We construct a $N(\varepsilon)$-dimensional test space $S_{N} \subset H_{0}^{1}(\Omega)$. For each of the $M(\varepsilon)$ high-contrast inclusions $D_{i}^{\varepsilon}$ (which have size $\mathcal{O}(\varepsilon)$), we construct a test function $\phi_{i}^{\varepsilon}$ that is exactly the local null mode (e.g., $\phi_{i}^{\varepsilon} = 1$ inside $D_{i}^{\varepsilon}$) and decays smoothly to zero over a distance of $\mathcal{O}(\varepsilon)$ within the background matrix $\Omega_{0}^{\varepsilon}$.

Since $\phi_{i}^{\varepsilon}$ is a strict null mode inside the inclusion, the high-contrast energy penalty vanishes: $\int_{D_{i}^{\varepsilon}} |\nabla \phi_{i}^{\varepsilon}|^{2} dx = 0$. In the background domain, the gradient scales as $|\nabla \phi_{i}^{\varepsilon}| \sim \mathcal{O}(\varepsilon^{-1})$, and the volume of the support is $\mathcal{O}(\varepsilon^{d})$. Thus, the background energy is
\begin{equation}
    a_{0}(\phi_{i}^{\varepsilon}, \phi_{i}^{\varepsilon}) = \int_{\Omega_{0}^{\varepsilon}} |\nabla \phi_{i}^{\varepsilon}|^{2} dx \sim \varepsilon^{d}(\varepsilon^{-1})^{2} = \mathcal{O}(\varepsilon^{d-2}).
\end{equation}
Conversely, the mass is heavily weighted by the inclusion:
\begin{equation}
    m^{\varepsilon}(\phi_{i}^{\varepsilon}, \phi_{i}^{\varepsilon}) \ge \eta \int_{D_{i}^{\varepsilon}} |\phi_{i}^{\varepsilon}|^{2} dx \sim \eta \varepsilon^{d}.
\end{equation}
Evaluating the Rayleigh quotient for any $u \in S_{N}$:
\begin{equation}
    R^{\varepsilon}(u) \le \frac{\mathcal{O}(\varepsilon^{d-2})}{\eta \mathcal{O}(\varepsilon^{d})} = \mathcal{O}(\varepsilon^{-2}\eta^{-1}).
\end{equation}
By the Min-Max principle, minimizing over all $n$-dimensional subspaces yields $\lambda_{n}^{\varepsilon} \le \mathcal{O}(\varepsilon^{-2}\eta^{-1})$ for all $n \le N(\varepsilon)$.

\textbf{Step 2: lower bound for $n = N(\varepsilon)+1$.}
To establish a rigorous lower bound, we employ the Courant-Fischer Max-Min formulation:
\begin{equation}
    \lambda_{N+1}^{\varepsilon} = \max_{\dim(S)=N(\varepsilon)} \min_{u \perp S, u \ne 0} R^{\varepsilon}(u).
\end{equation}
We choose the test subspace $S$ to be the exact $N$-dimensional global topological null space. Thus, for any valid test function $u^{*} \perp S$, its $L^{2}$-projection onto the local kernel of every microscopic inclusion $D_{i}^{\varepsilon}$ is strictly zero. This orthogonality constraint leads to two distinct coercivity bounds: 

\textbf{Local coercivity:} Inside the inclusions, the scaled local Poincaré inequality holds:
\begin{equation}
    \int_{D^{\varepsilon}} |\nabla u^{*}|^{2} dx \ge \frac{C_{1}}{\varepsilon^{2}} \int_{D^{\varepsilon}} |u^{*}|^{2} dx.
\end{equation}

\textbf{Global patch coercivity:} Because $u^{*}$ is constrained by having zero mean on a dense periodic array of inclusions (separated by $\mathcal{O}(\varepsilon)$), the total gradient globally controls the total mass via the multiscale Poincaré-Wirtinger inequality:
\begin{equation}
    \int_{\Omega_{0}^{\varepsilon}} |\nabla u^{*}|^{2} dx + \int_{D^{\varepsilon}} |\nabla u^{*}|^{2} dx \ge \frac{C_{2}}{\varepsilon^{2}} \left( \int_{\Omega_{0}^{\varepsilon}} |u^{*}|^{2} dx + \int_{D^{\varepsilon}} |u^{*}|^{2} dx \right). 
\end{equation}
Let $E_{0}, E_{1}$ be the gradient energies in $\Omega_{0}^{\varepsilon}$ and $D^{\varepsilon}$, and $M_{0}, M_{1}$ be the corresponding $L^{2}$ masses. The numerator of the Rayleigh quotient can be algebraically rewritten and bounded by combining the two coercivity estimates:
\begin{equation}
    \text{Numerator} = E_{0} + \eta E_{1} = (E_{0} + E_{1}) + (\eta - 1)E_{1} \ge \frac{C_{2}}{\varepsilon^{2}}(M_{0} + M_{1}) + \frac{C_{1}(\eta - 1)}{\varepsilon^{2}}M_{1}.
\end{equation}
Let $C_{\min} = \min(C_{1}, C_{2}) > 0$. Factoring out $\frac{C_{\min}}{\varepsilon^{2}}$, we obtain
\begin{equation}
    \text{Numerator} \ge \frac{C_{\min}}{\varepsilon^{2}} [M_{0} + M_{1} + (\eta - 1)M_{1}] = \frac{C_{\min}}{\varepsilon^{2}} [M_{0} + \eta M_{1}].
\end{equation}
Interestingly, the bracketed term matches exactly the denominator of the multiscale Rayleigh quotient $R^{\varepsilon}(u^{*})$. Therefore, dividing by the denominator yields the following:
\begin{equation}
    R^{\varepsilon}(u^{*}) \ge \frac{C_{\min}}{\varepsilon^{2}} = \mathcal{O}(\varepsilon^{-2}).
\end{equation}
Since this bound strictly holds for any $u^{*} \perp S$, the Max-Min principle concludes that $\lambda_{N+1}^{\varepsilon} \ge \mathcal{O}(\varepsilon^{-2})$. This confirms that regardless of whether the eigenfunction's mass concentrates in the high-contrast inclusions or the background matrix, the energy universally scales as $\epsilon^{-2}$, preserving the  spectral gap.
\end{proof}

\subsection{Spectral completeness}
In this section, we demonstrate how the low-frequency topological modes, trapped within the spectral gap, asymptotically and fully converge to the entire spectrum of the homogenized operator as $\varepsilon \rightarrow 0$.

Under the geometric assumption of disconnected inclusions, the effective coefficients inherit distinct scalings from the microscopic contrast $\eta$. Since the mass density scales as $\mathcal{O}(\eta)$ in $D$ and $\mathcal{O}(1)$ in $\Omega_{0}$, the effective density is dominated by the inclusions:
\begin{equation}
    \rho^{*} = \frac{1}{|Y|} \int_{Y} \kappa(y) dy = \mathcal{O}(\eta).
\end{equation}
Conversely, because any macroscopic flux must traverse the background matrix $\Omega_{0}$ to pass between disconnected inclusions, the effective tensor $\kappa^{*}$ is restricted by the background matrix conductivity, yielding $\kappa^{*} = \mathcal{O}(1)$ \cite{zhikov2000extension}. This implies that the homogenized eigenvalues satisfy:
\begin{equation}
    \lambda_{n}^{0} \sim \frac{\kappa^{*}}{\rho^{*}} = \mathcal{O}(\eta^{-1}), \quad \forall n \in \mathbb{N}.
\end{equation}

\begin{theorem}
\label{thm:spectral_completeness}
As $\varepsilon \to 0$, the topological eigenspace with gradually expanding dimensions $S_{N(\varepsilon)}^\varepsilon = \text{span}\{u_1^\varepsilon, \dots, u_{N(\varepsilon)}^\varepsilon\}$ asymptotically exhausts the infinite-dimensional spectral space of the homogenized operator. Specifically, for any given macroscopic homogenized eigenpair $(\lambda_n^0, u_n^0)$, there exists a scale $\varepsilon_0(n) > 0$ such that for all $\varepsilon < \varepsilon_0(n)$, the corresponding multiscale eigenpair strictly belongs to the topological low-frequency band (i.e., $n \le N(\varepsilon)$) and converges:
\begin{equation}
    \lim_{\varepsilon \to 0} \lambda_n^\varepsilon = \lambda_n^0, \quad \text{and} \quad \lim_{\varepsilon \to 0} |u_n^\varepsilon - u_n^0|_{L^2(\Omega)} = 0.
\end{equation}
\end{theorem}

\begin{proof}
The dimension of the topological eigenspace is governed by the total number of local null modes. Since the number of microscopic inclusions $M(\varepsilon)$ scales inversely with the unit cell volume, the global null space dimension exhibits an algebraic blow-up:
\begin{equation}
    N(\epsilon) = R \cdot M(\epsilon) = \mathcal{O}(|\Omega|\epsilon^{-d}) \implies \lim_{\epsilon \to 0} N(\epsilon) = \infty.
\end{equation}

Fix an arbitrary macroscopic mode index $n \in \mathbb{N}$. Due to the strict divergence of $N(\varepsilon)$, there is a scale $\varepsilon_0(n) > 0$ such that for all $\varepsilon < \varepsilon_0(n)$, the inequality $n \le N(\varepsilon)$ is satisfied. This geometric constraint guarantees that the $n$-th multiscale eigenpair $(\lambda_n^\varepsilon, u_n^\varepsilon)$ is trapped within the topological low-frequency band, yielding
\begin{equation}
    u_n^\varepsilon \in S_{N(\varepsilon)}^\varepsilon, \quad \forall \varepsilon < \varepsilon_0(n).
\end{equation}

Applying the spectral convergence established in \cref{lemma:homo_spec_conver}, we immediately obtain
\begin{equation}
    \lambda_n^\varepsilon \to \lambda_n^0, \quad \text{and} \quad u_n^\varepsilon \rightharpoonup u_n^0 \text{ weakly in } H_0^1(\Omega) \text{ (and strongly in } L^2(\Omega)).
\end{equation}
Since the index $n \in \mathbb{N}$ is arbitrary, this confirms that for any target eigenpair of the macroscopic homogenized operator, the associated multiscale sequence eventually falls into the topological subspace $S_{N(\varepsilon)}^\varepsilon$ and converges exactly to it. Consequently, the expanding topological subspace asymptotically recovers the complete infinite-dimensional spectrum of $\mathcal{A}^0$ as $\varepsilon \to 0$, concluding the proof.
\end{proof}

In summary, the robust band separation and spectral completeness established in this section reveal the profound physical and mathematical significance of the topological spectral gap. The eigenspace below the spectral gap, corresponding to the topological low-frequency band, can be intrinsically understood as the carrier of macroscopic information. In the asymptotic limit of vanishing periodicity ($\varepsilon \to 0$), this expanding topological subspace continuously evolves to span the entire infinite-dimensional spectral space of the macroscopic homogenized operator, thereby fully preserving the macroscale effective physics. In contrast, the high-frequency eigenmodes beyond the spectral gap capture highly oscillatory microscopic information and localized resonances. As $\varepsilon \to 0$, the diverging energy cost forces these microscopic oscillations to asymptotically vanish from the homogenized macroscopic description. Consequently, the spectral gap serves as a fundamental mathematical demarcation between macroscopic wave behaviors and microscopic details, providing a solid theoretical justification for multiscale model reduction and coarse-space design.

\section{Numerical experiments}
\label{sec:num_ex}
In this section, we present a series of numerical experiments designed to validate the theoretical and topological frameworks established in the preceding sections. Our primary objectives are to demonstrate the existence of the topological spectral gap and to illustrate its universal dependence on the physical properties of the operator, the local kernel dimensions, and the topological invariants of the high-contrast inclusions. Furthermore, we numerically verify the asymptotic spectral decoupling theory and the multiscale Weyl's law under extreme high-contrast limits. Finally, we confirm the robust asymptotic band separation and spectral completeness of the low-frequency topological modes in the limit of vanishing periodicity.

\subsection{Numerical verification in 2D}
Our first set of numerical experiments focuses on the two-dimensional case. We consider a square domain $\Omega = (0, 1)^2$ containing $M = 4$ disconnected high-contrast inclusions $D_i$. The contrast parameter is a sufficiently large value $\eta = 10^5$ to clearly separate the low-frequency modes from the rest of the spectrum.

According to our unified variational framework, the number of eigenvalues $\lambda_n$ that scale as $\mathcal{O}(\eta^{-1})$ (i.e. those below the spectral gap) is given by $N = R \times M$, where $R$ is the dimension of the kernel of the local differential operator. We verify this ``spectral gap" by examining three distinct physical problems: the scalar diffusion equation, the linear elasticity equation, and the fourth-order biharmonic equation.

\begin{figure}[htbp]
    \centering
    \begin{subfigure}[b]{0.32\textwidth}
        \centering
        \includegraphics[width=\textwidth]{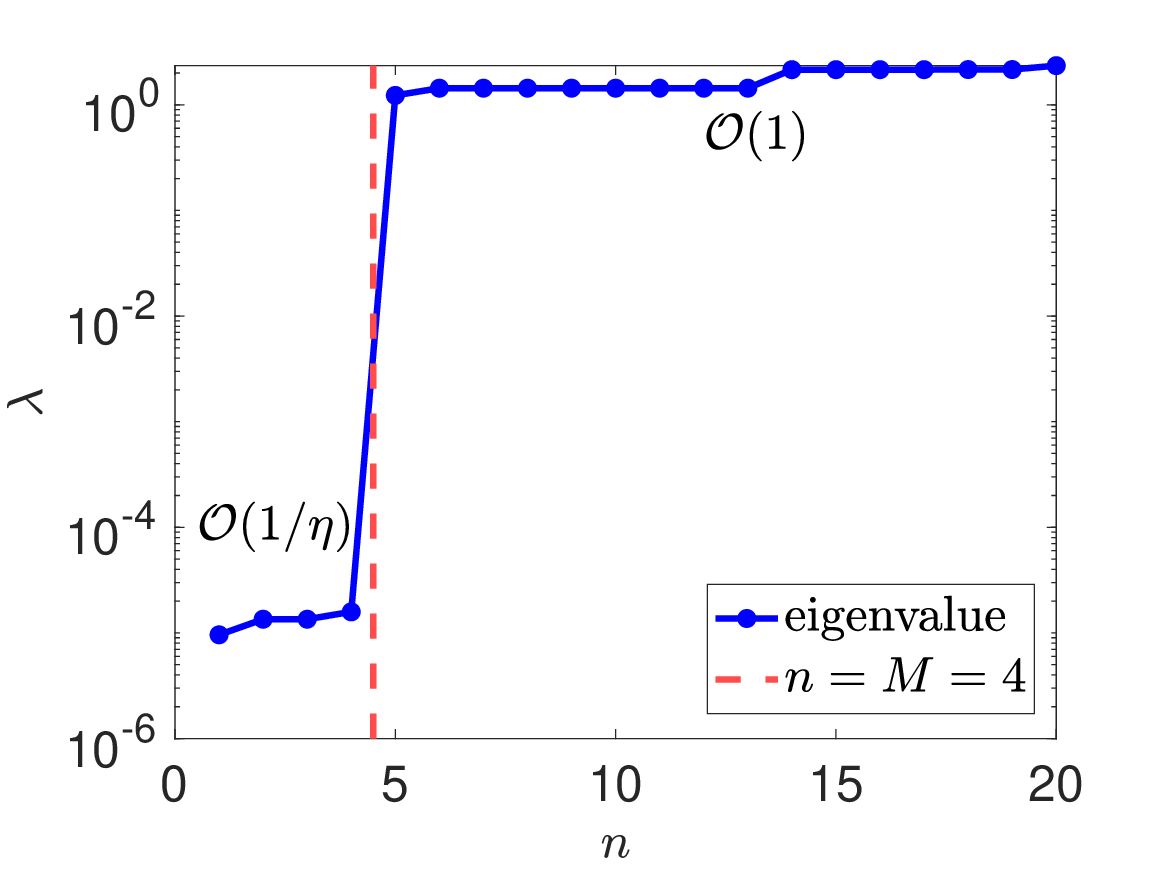}
        \caption{Scalar diffusion ($R=1$)}
        \label{fig:2d_diff}
    \end{subfigure}
    \hfill
    \begin{subfigure}[b]{0.32\textwidth}
        \centering
        \includegraphics[width=\textwidth]{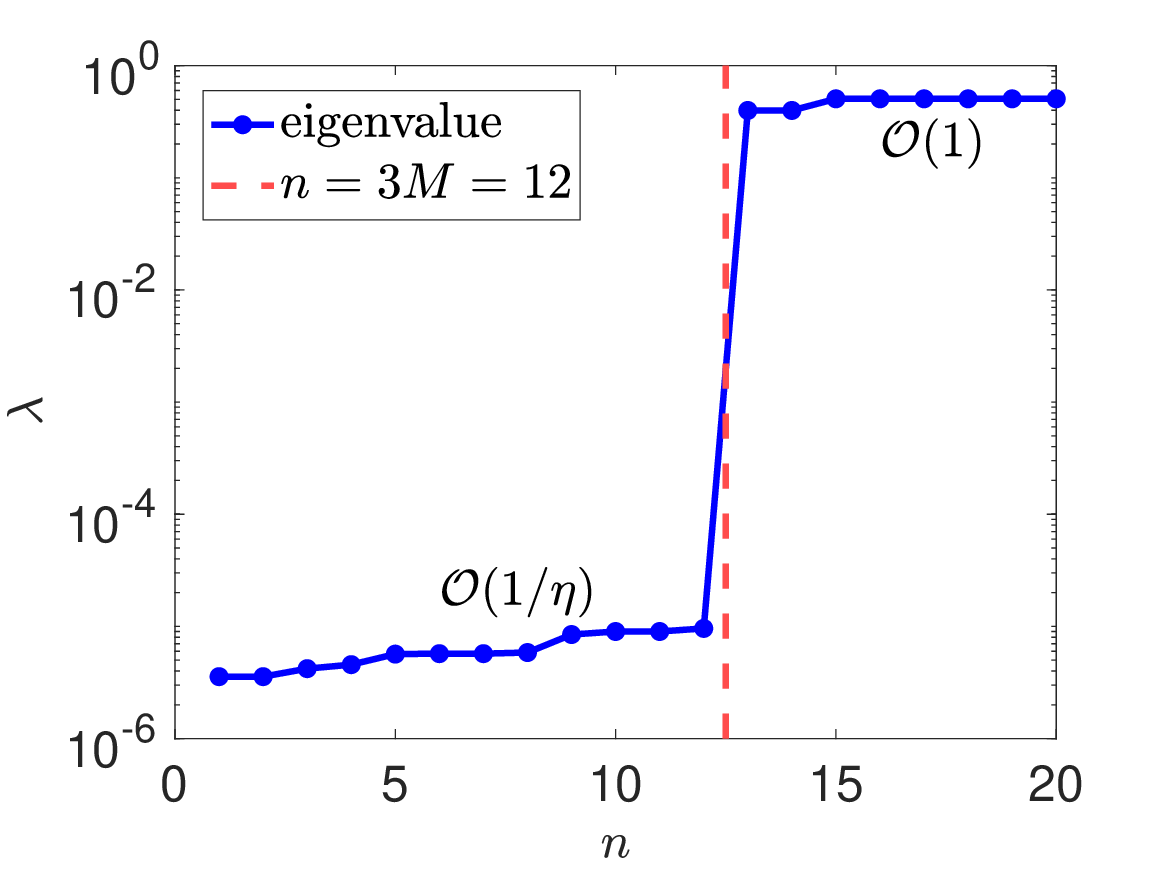}
        \caption{Linear elasticity ($R=3$)}
        \label{fig:2d_elas}
    \end{subfigure}
    \hfill
    \begin{subfigure}[b]{0.32\textwidth}
        \centering
        \includegraphics[width=\textwidth]{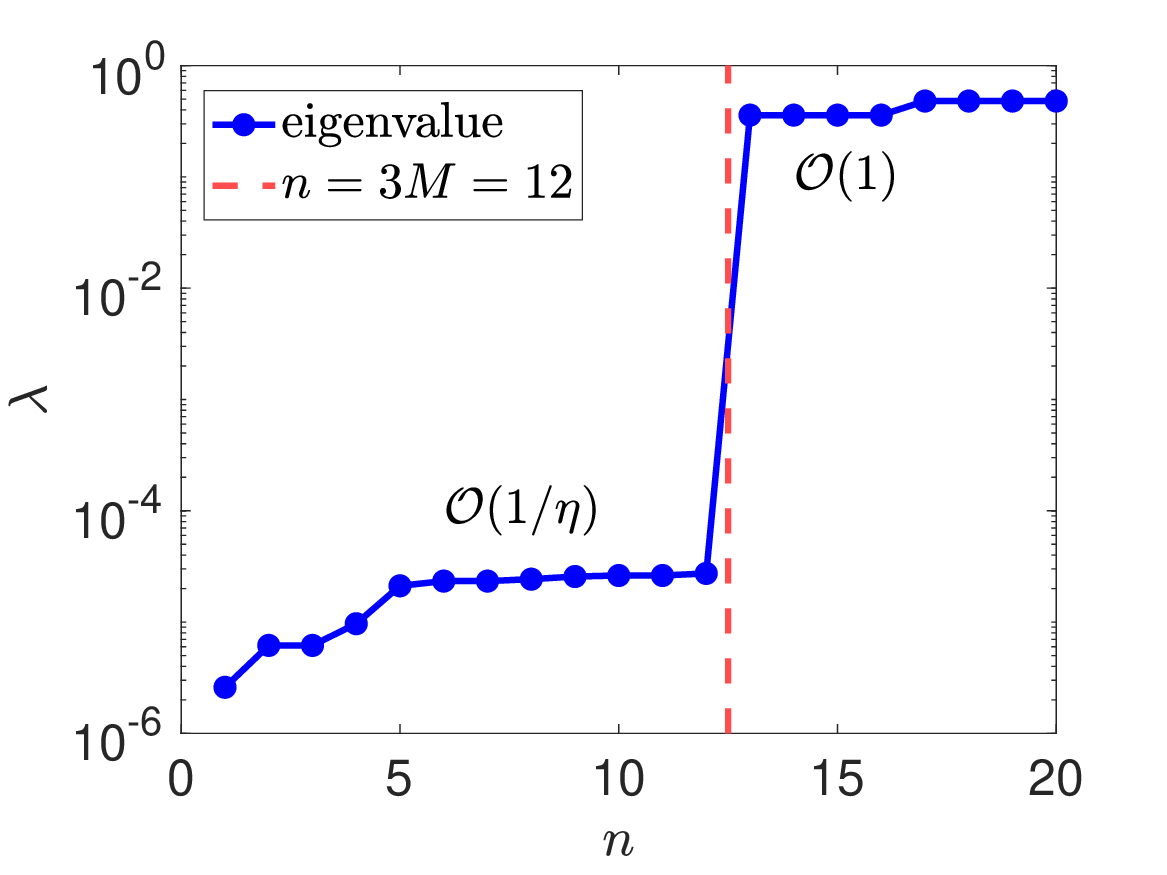}
        \caption{Fourth-order plate ($R=3$)}
        \label{fig:2d_plate}
    \end{subfigure}
\caption{The first 20 computed eigenvalues $\lambda_n$ for three different physical operators with $M=4$ high-contrast inclusions. The vertical dashed lines indicate the theoretically predicted position of the spectral gap at $n = R \times M + 1$.}
\label{fig:2d_comparison}
\end{figure}

The first 20 computed eigenvalues are presented in  \cref{fig:2d_comparison}. The results are interpreted as follows:
\begin{itemize}
    \item \textbf{Scalar diffusion:} The local kernel consists of constant functions, thus $R=1$. For $M=4$ inclusions, we expect $N = 1 \times 4 = 4$ small eigenvalues. As shown in \cref{fig:2d_diff}, a significant jump in the eigenvalue magnitude occurs between $\lambda_4$ and $\lambda_5$.
    \item \textbf{Linear elasticity:} The kernel consists of rigid body motions (two translations and one rotation), which yields $R=3$. In this case, $N = 3 \times 4 = 12$ eigenvalues are expected below the spectral gap. \cref{fig:2d_elas} confirms that $\lambda_{12}$ is close to zero while $\lambda_{13}$ is of the order $\mathcal{O}(1)$.
    \item \textbf{Fourth-order plate:} The kernel of the biharmonic operator is the space of linear polynomials $\{1, x, y\}$, which also has dimension $R=3$. Consequently, the spectral gap appears after $N = 12$ eigenvalues, which is consistent with the numerical result in \cref{fig:2d_plate}.
\end{itemize}

These 2D results provide strong evidence that the position of the topological spectral gap is universally determined by the product of the number of inclusions and the dimension of the local kernel, regardless of the specific operator type.

\subsection{Numerical verification in 3D}
To further demonstrate the universality of the proposed unified framework, we extend our numerical experiments to the three-dimensional case. We consider a cubic domain $\Omega = (0, 1)^3$ containing $M = 3$ disconnected high-contrast inclusions. The contrast ratio remains $\eta = 10^5$. 

In 3D, the dimension of the local kernel $R$ increases for vector-valued and higher-order problems. We examine the eigenvalue distributions for the same three classes of equations to verify if the spectral gap still obeys the rule $N = R \times M$.

\begin{figure}[htbp]
    \centering
    \begin{subfigure}[b]{0.32\textwidth}
        \centering
        \includegraphics[width=\textwidth]{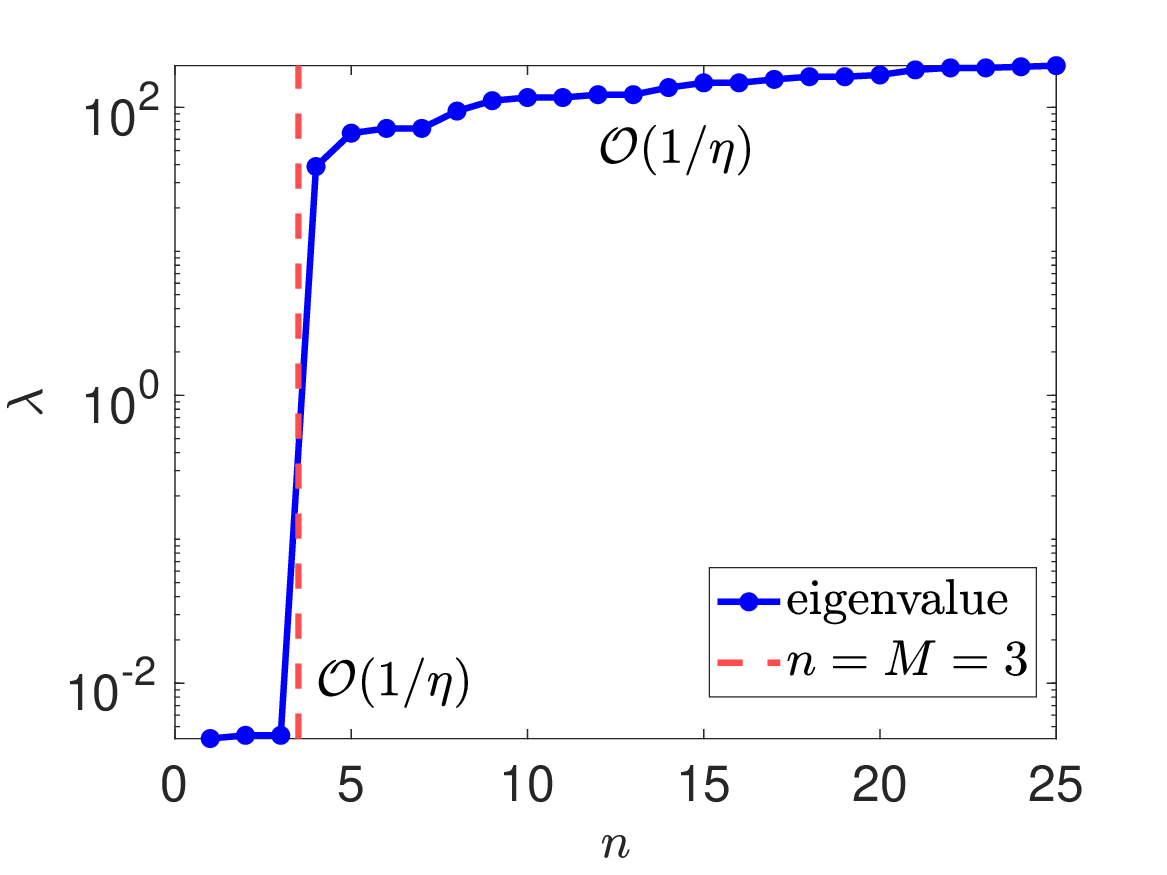}
        \caption{Scalar diffusion ($R=1$)}
        \label{fig:3d_diff}
    \end{subfigure}
    \hfill
    \begin{subfigure}[b]{0.32\textwidth}
        \centering
        \includegraphics[width=\textwidth]{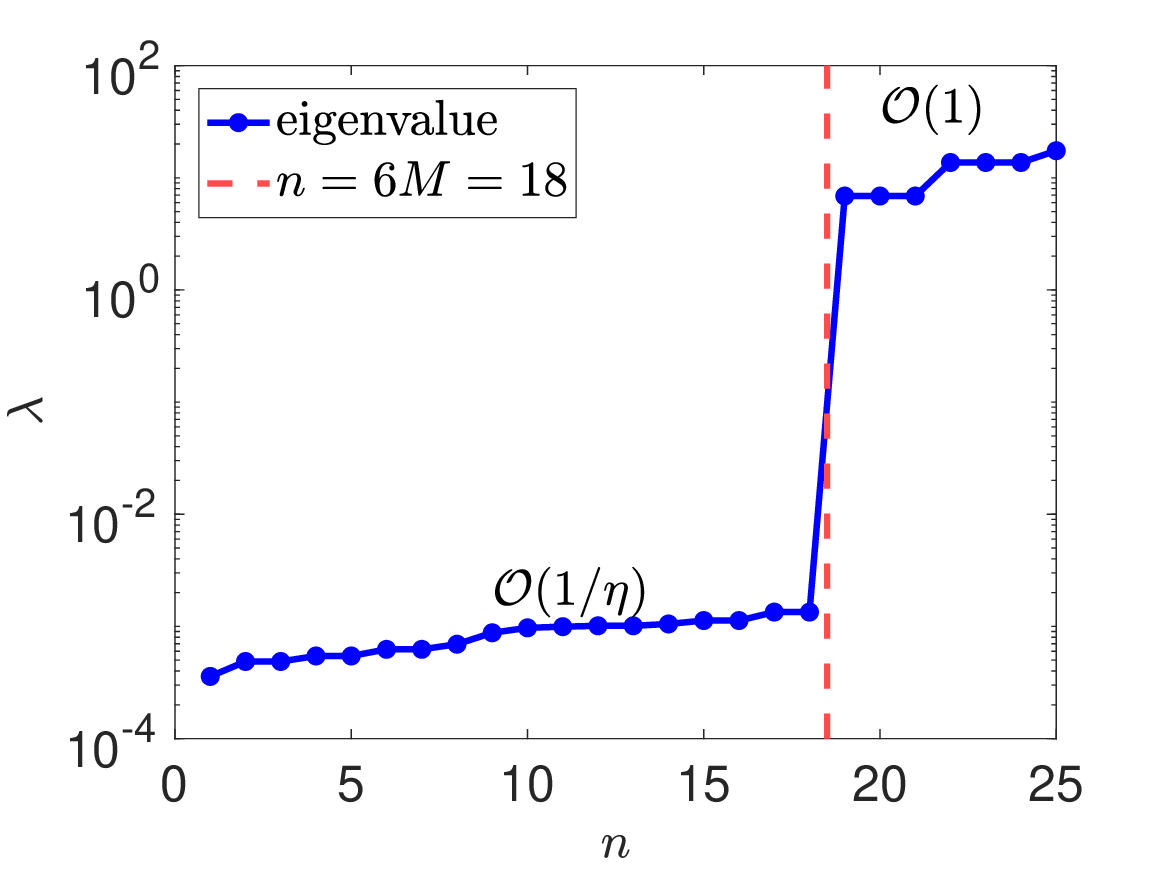}
        \caption{Linear elasticity ($R=6$)}
        \label{fig:3d_elas}
    \end{subfigure}
    \hfill
    \begin{subfigure}[b]{0.32\textwidth}
        \centering
        \includegraphics[width=\textwidth]{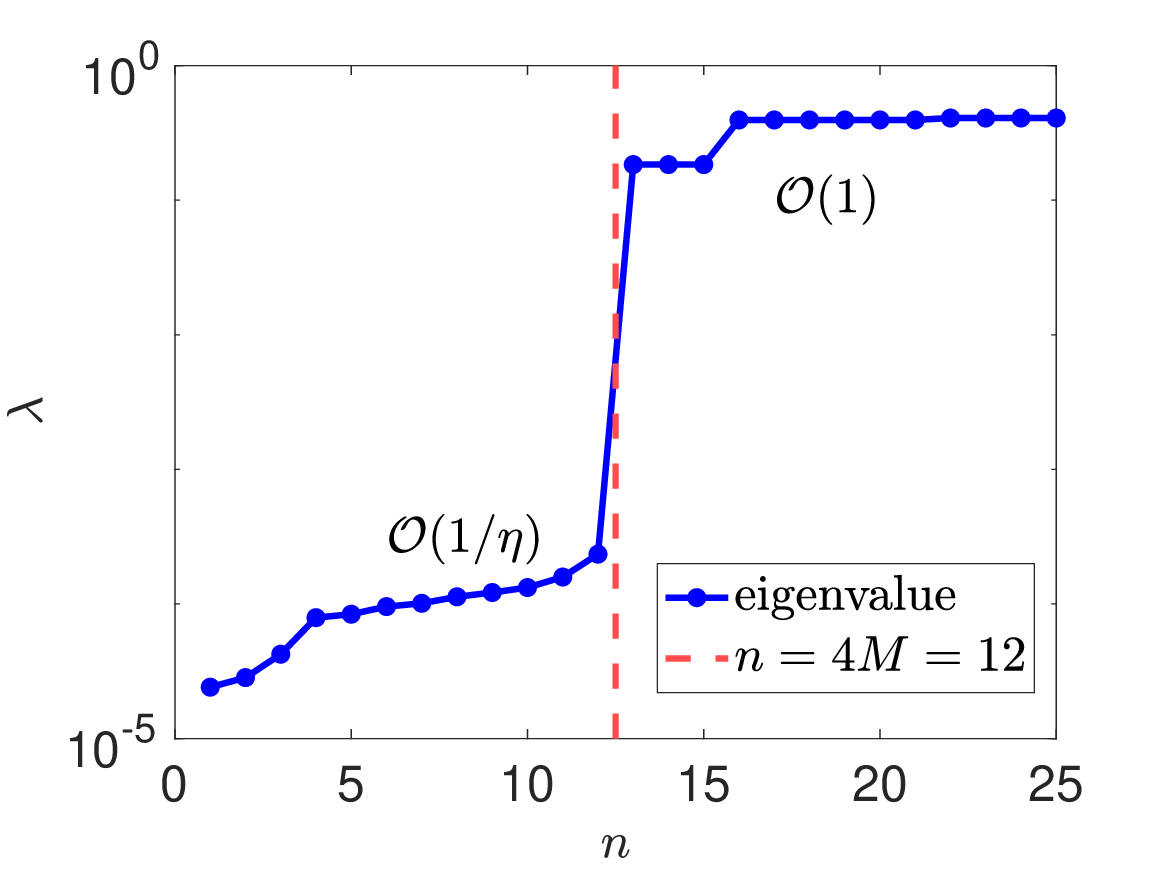}
        \caption{Fourth-order plate ($R=4$)}
        \label{fig:3d_plate}
    \end{subfigure}
    \caption{The first 25 computed eigenvalues $\lambda_n$ in 3D with $M=3$ inclusions. The vertical dashed lines represent the predicted spectral gap positions based on the dimension of the 3D local kernels.}
    \label{fig:3d_comparison}
\end{figure}

The numerical results illustrated in \cref{fig:3d_comparison} confirm the following:
\begin{itemize}
    \item \textbf{Scalar diffusion:} The kernel still consists of constant functions, so $R=1$. With $M=3$, the gap remains at $N=3$, as shown in \cref{fig:3d_diff}.
    \item \textbf{Linear elasticity:} In 3D, the space of rigid body motions is spanned by three translations and three rotations, leading to $R=6$. Our theory predicts $N = 6 \times 3 = 18$ small eigenvalues. \cref{fig:3d_elas} clearly shows that the first 18 eigenvalues are near zero, with a sharp jump at $n=19$.
    \item \textbf{Fourth-order plate:} The kernel of the 3D biharmonic operator is the space of linear polynomials $\{1, x, y, z\}$, which has dimension $R=4$. Consequently, we observe $N = 4 \times 3 = 12$ eigenvalues below the spectral gap in \cref{fig:3d_plate}.
\end{itemize}

The perfect agreement between the 3D numerical results and our theoretical predictions confirms that the topological spectral gap is a fundamental property determined by the interaction between the number of inclusions $M$ and the local kernel dimension $R$.

\subsection{Verification of topological spectral gap: Maxwell and grad-div}
In this numerical experiment, we investigate the spectral behavior of systems with infinite-dimensional kernels, specifically the time-harmonic Maxwell's equations and grad-div operator, as discussed in \cref{sec:derham}. Unlike the previous cases governed by finite-dimensional local kernels (e.g., constants or polynomials), the spectral gaps for these systems are determined by the topological invariants of the inclusions—the Betti numbers $\beta_{k}$. To filter out the redundant infinite-dimensional kernels, we perform computations in the gauge-fixed spaces, which restricts the null space to the finite-dimensional harmonic spaces, maintaining a high contrast ratio of $\eta=10^{5}$.
To validate the topological spectral gap theorem (\cref{thm:homo_spectral_gap}) comprehensively, we design two test cases with varying topological complexities.

\textbf{Test case 1: simple uniform topology.} We first consider a cubic domain $\Omega=(0,1)^{3}$ with $M=2$ high-contrast inclusions, each possessing a simple non-trivial topology of $\beta_{k}(D_{i})=1$:
\begin{itemize}
    \item For the Maxwell system (1-Forms), we select two inclusions with the first Betti number $\beta_{1}(D_{i})=1$ (e.g., two standard tori).
    \item For the grad-div operator (2-Forms), we select two inclusions with the second Betti number $\beta_{2}(D_{i})=1$ (e.g., two shells with a single internal cavity).
\end{itemize}

\begin{figure}[htbp]
    \centering
    \begin{subfigure}[b]{0.48\textwidth}
        \centering
        \includegraphics[width=\textwidth]{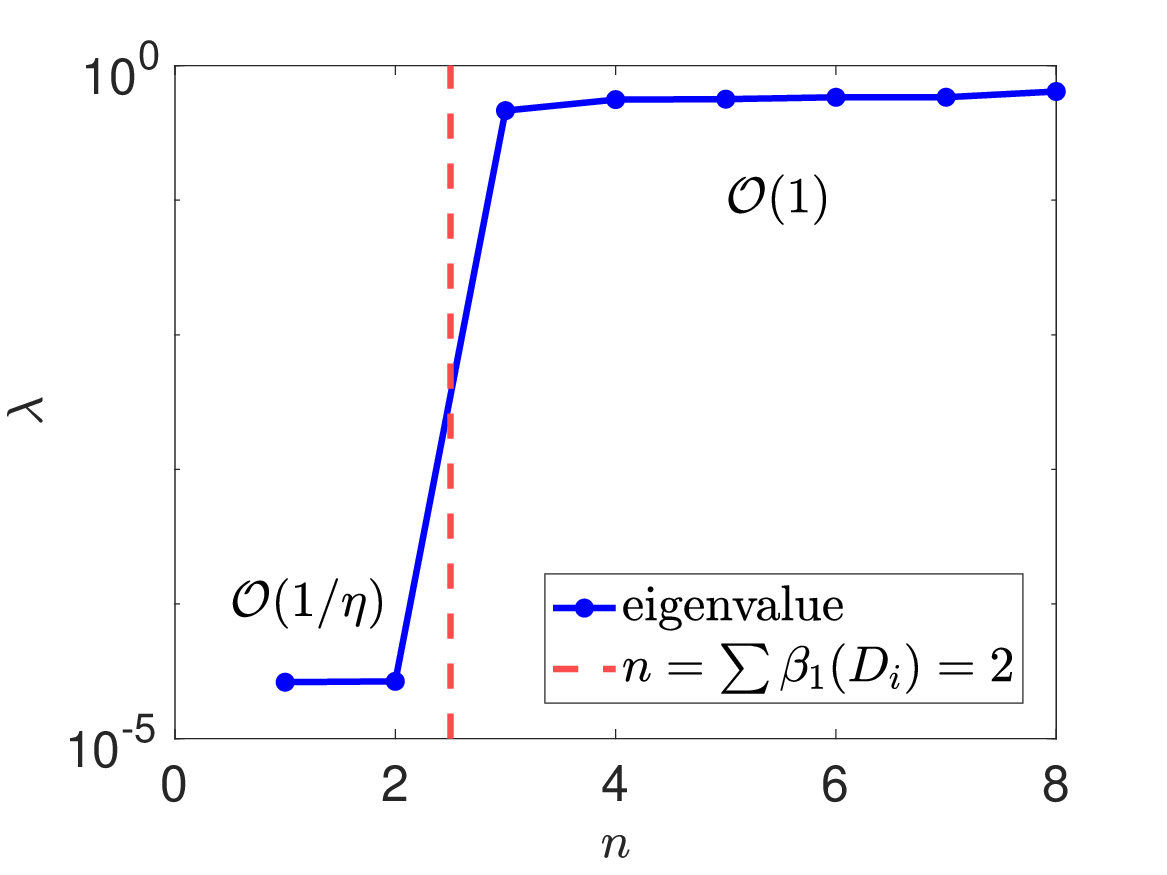}
        \caption{Maxwell's equations }
        \label{fig:3d_maxwell}
    \end{subfigure}
    \hfill
    \begin{subfigure}[b]{0.48\textwidth}
        \centering
        \includegraphics[width=\textwidth]{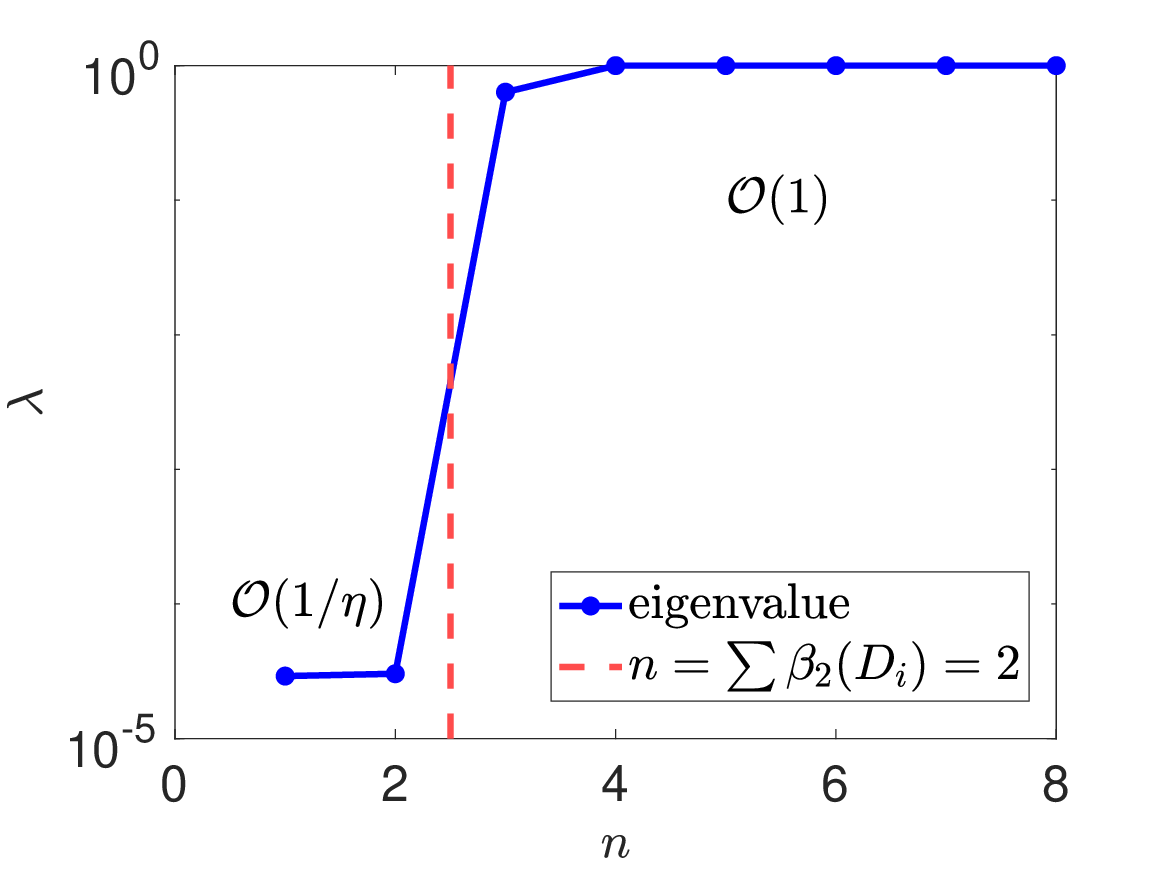}
        \caption{Grad-div operator}
        \label{fig:3d_darcy}
    \end{subfigure}
    \caption{The first 8 computed eigenvalues for Maxwell and grad-div systems in 3D. In both cases, with $M=2$ and local Betti numbers $\beta_k = 1$, the topological spectral gap emerges at $n = \sum_{i} \beta_k(D_i) + 1 = 3$.}
    \label{fig:homological_verification}
\end{figure}

As shown in \cref{fig:homological_verification}, exactly $N_{topo} = 1+1=2$ eigenvalues collapse to zero, scaling as $\mathcal{O}(\eta^{-1})$. The spectral gap emerges at $n=3$, bounded below by a generalized coercivity constant independent of $\eta$.

\textbf{Test case 2: heterogeneous complex topology.} To further validate the topological spectral gap theorem under heterogeneous and higher-order topological complexities, we construct a second test case with $M=2$ inclusions possessing mixed Betti numbers: $\beta_{k}(D_{1}) = 2$ and $\beta_{k}(D_{2}) = 3$.
\begin{itemize}
    \item For the Maxwell system, this corresponds to structures with multiple topological tunnels (e.g., a figure-eight torus and a triple torus).
    \item For the grad-div operator, this is realized by solid blocks containing multiple completely isolated internal cavities (two and three cavities, respectively).
\end{itemize}

\begin{figure}[htbp]
    \centering
    \begin{subfigure}{0.48\textwidth}
        \centering
        \includegraphics[width=\linewidth]{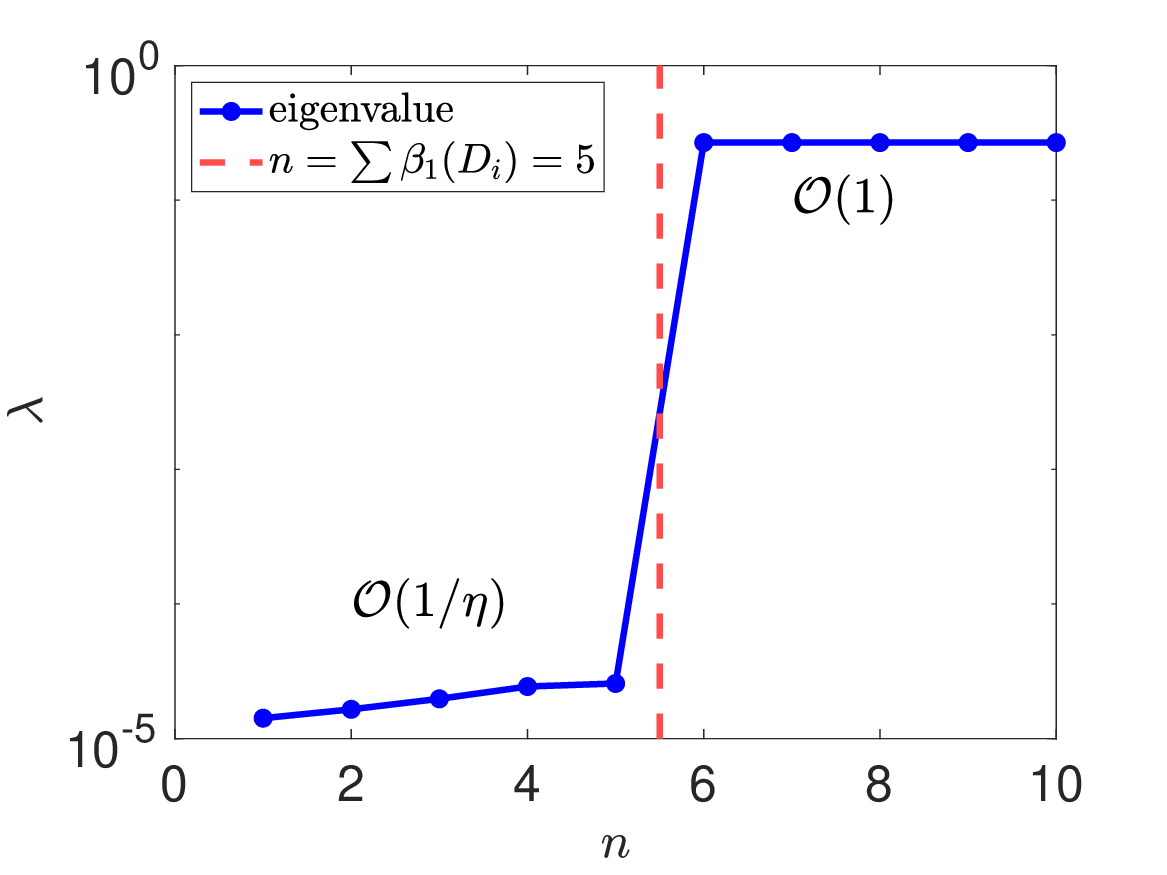}
        \caption{Maxwell's equations }
    \end{subfigure}\hfill
    \begin{subfigure}{0.48\textwidth}
        \centering
        \includegraphics[width=\linewidth]{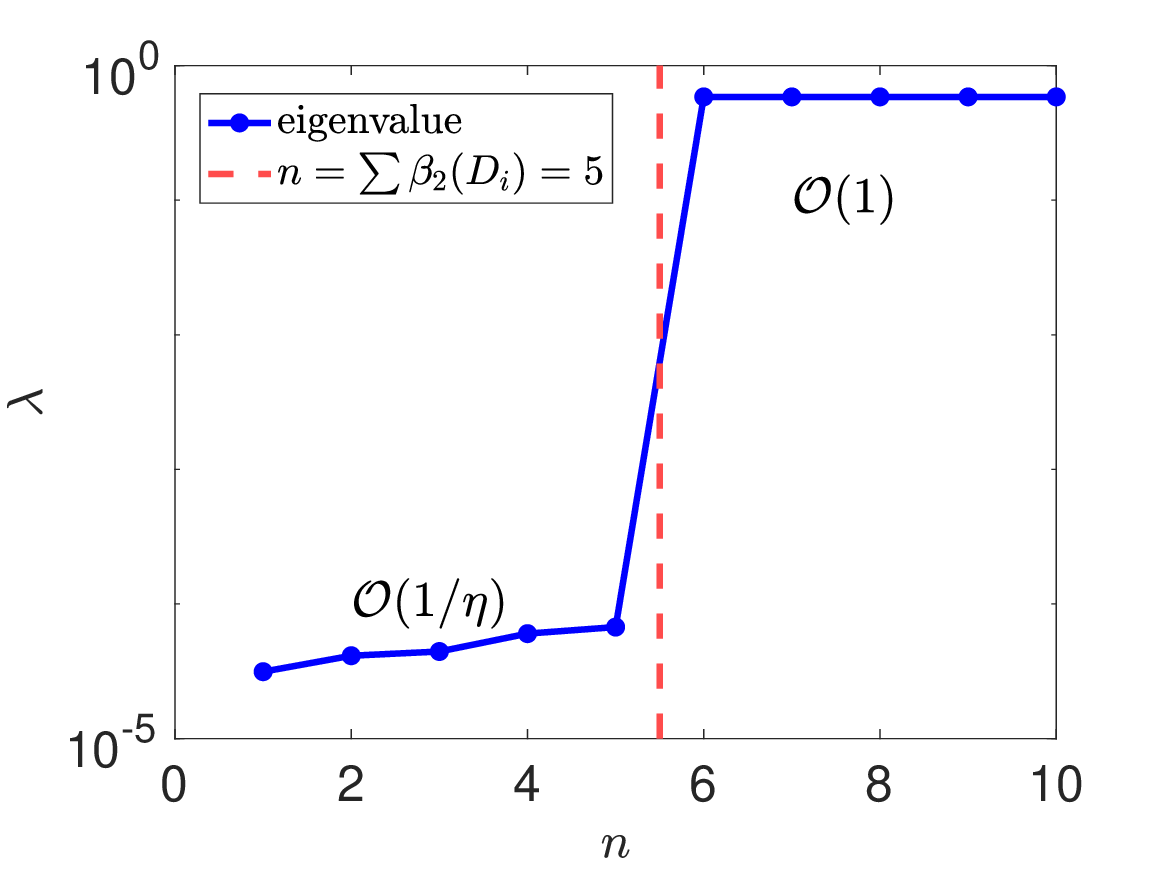}
        \caption{Grad-div operator}
    \end{subfigure}
    \caption{The first 10 computed eigenvalues for Maxwell and grad-div systems with highly complex heterogeneous topologies. With the local Betti numbers being 2 and 3, the topological spectral gap emerges at $n = \sum \beta_k + 1 = 6$.}
    \label{fig:complex_topology}
\end{figure}

According to \cref{thm:homo_spectral_gap}, the topological invariant is $N_{topo} = \sum_{i=1}^{2}\beta_{k}(D_{i}) = 2 + 3 = 5$. The first 10 computed eigenvalues are presented in \cref{fig:complex_topology}. For both physical systems, exactly 5 eigenvalues fall into the $\mathcal{O}(\eta^{-1})$ band. 
The generalized coercivity bound ensures that the spectral gap emerges at the 6-th mode ($n=6$), meaning that this eigenvalue is of order $\mathcal{O}(1)$ and independent of the contrast parameter $\eta$.

These results provide strong numerical evidence for the topological rule $N_{topo}=\sum_{i}\beta_{k}(D_{i})$. They confirm that the spectral behavior is exclusively governed by the total topological invariant of the system, remaining completely robust against variations in the geometric complexity or local topological features of individual inclusions.

\subsection{Spectral decoupling and the multiscale Weyl's law}
In this subsection, we numerically verify the asymptotic spectral decoupling theory proposed in \cref{sec:spec_decoupling}. The experiment is conducted on a 2D square domain $\Omega$ containing $M=4$ disconnected high-contrast inclusions $D_i$. To explicitly demonstrate the asymptotic convergence established in \cref{theorem:spec_decou}, we systematically vary the contrast parameter across several orders of magnitude, specifically $\eta \in \{10^1, 10^2, 10^3, 10^4\}$. According to our theoretical analysis, as the contrast $\eta \rightarrow \infty$, the spectrum of the full multiscale system decouples and asymptotically approaches the union of the macroscopic background Dirichlet spectrum $\sigma(\mathcal{L}_{\Omega_0}^{Dir})$ and the internal Neumann spectrum of the inclusions $\sigma(\mathcal{L}_{D}^{Neu})$.

\begin{figure}[htbp]
    \centering
    \begin{subfigure}{0.48\textwidth}
        \includegraphics[width=\textwidth]{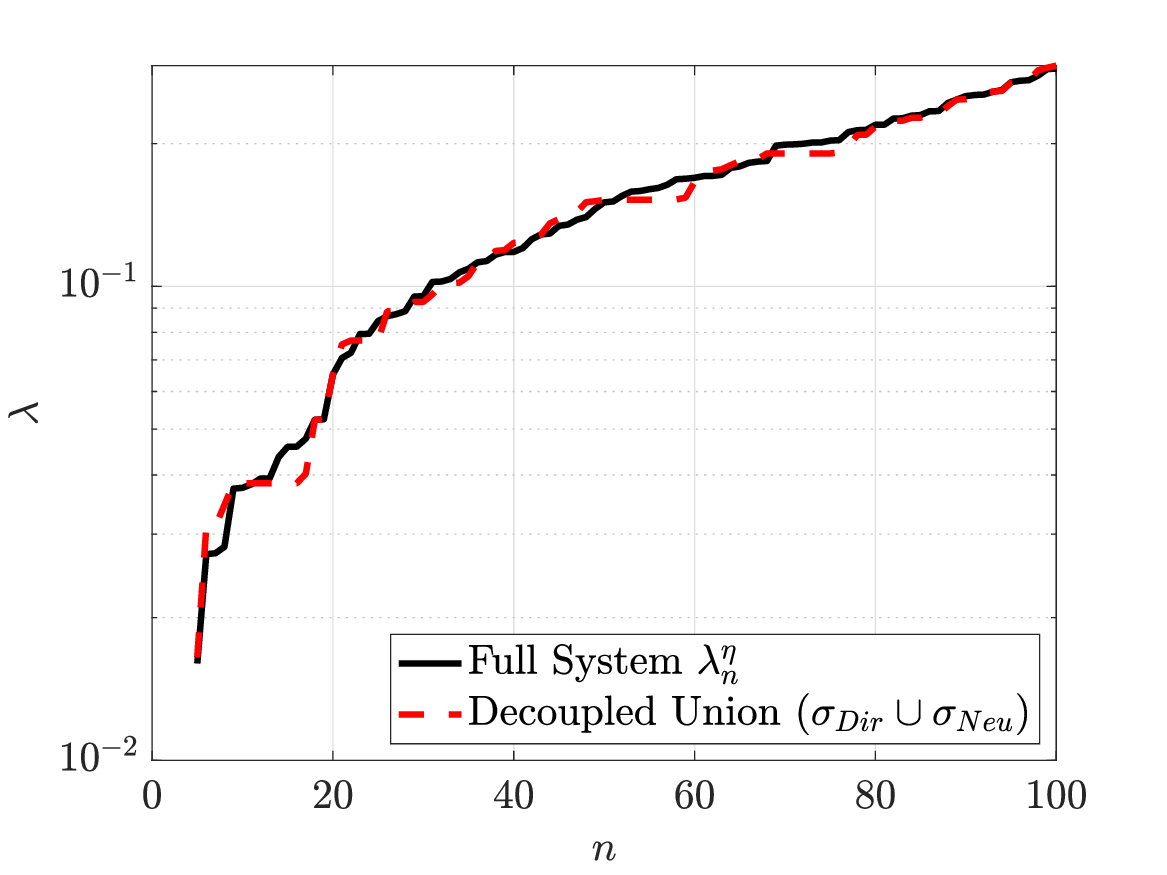}
        \caption{$\eta = 10^1$}
        \label{fig:sd1}
    \end{subfigure}
    \hfill
    \begin{subfigure}{0.48\textwidth}
        \includegraphics[width=\textwidth]{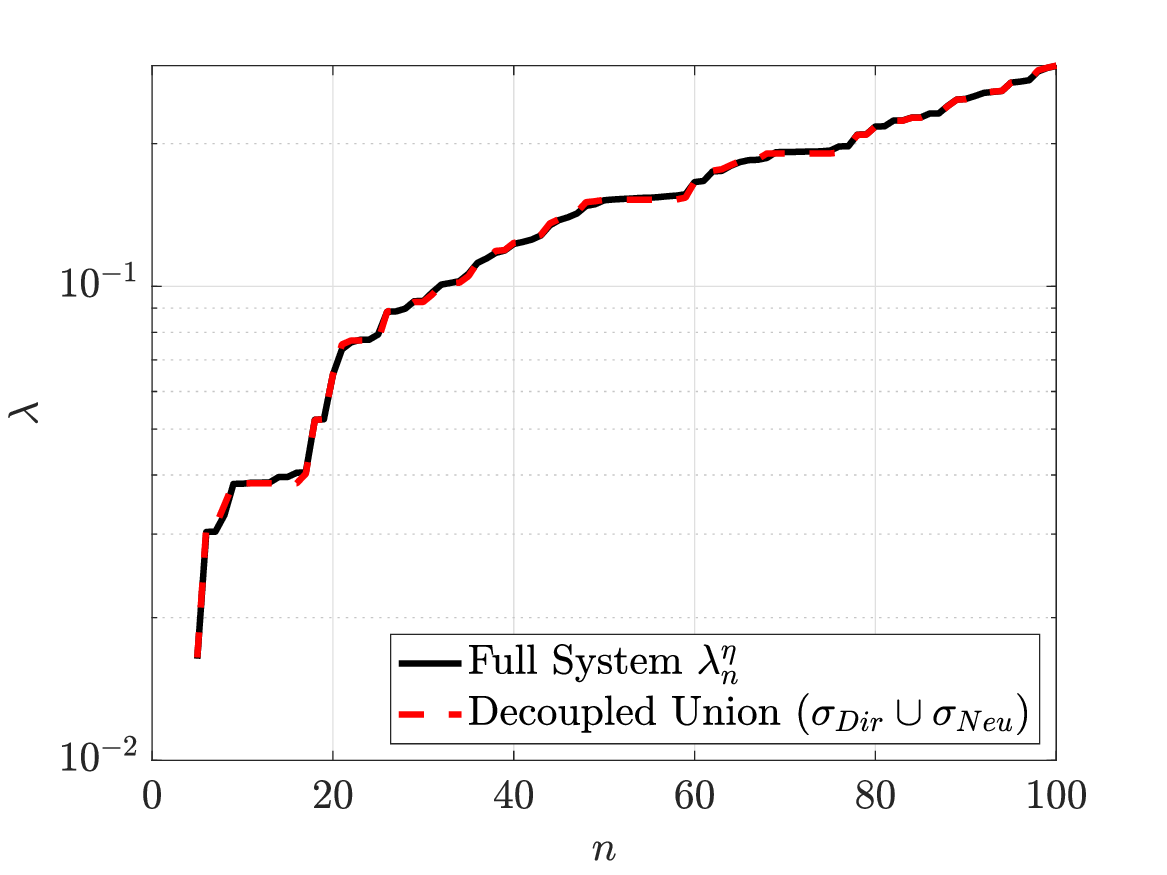}
        \caption{$\eta = 10^2$}
        \label{fig:sd2}
    \end{subfigure}
    
    \vspace{0.5cm}
    
    \begin{subfigure}{0.48\textwidth}
        \includegraphics[width=\textwidth]{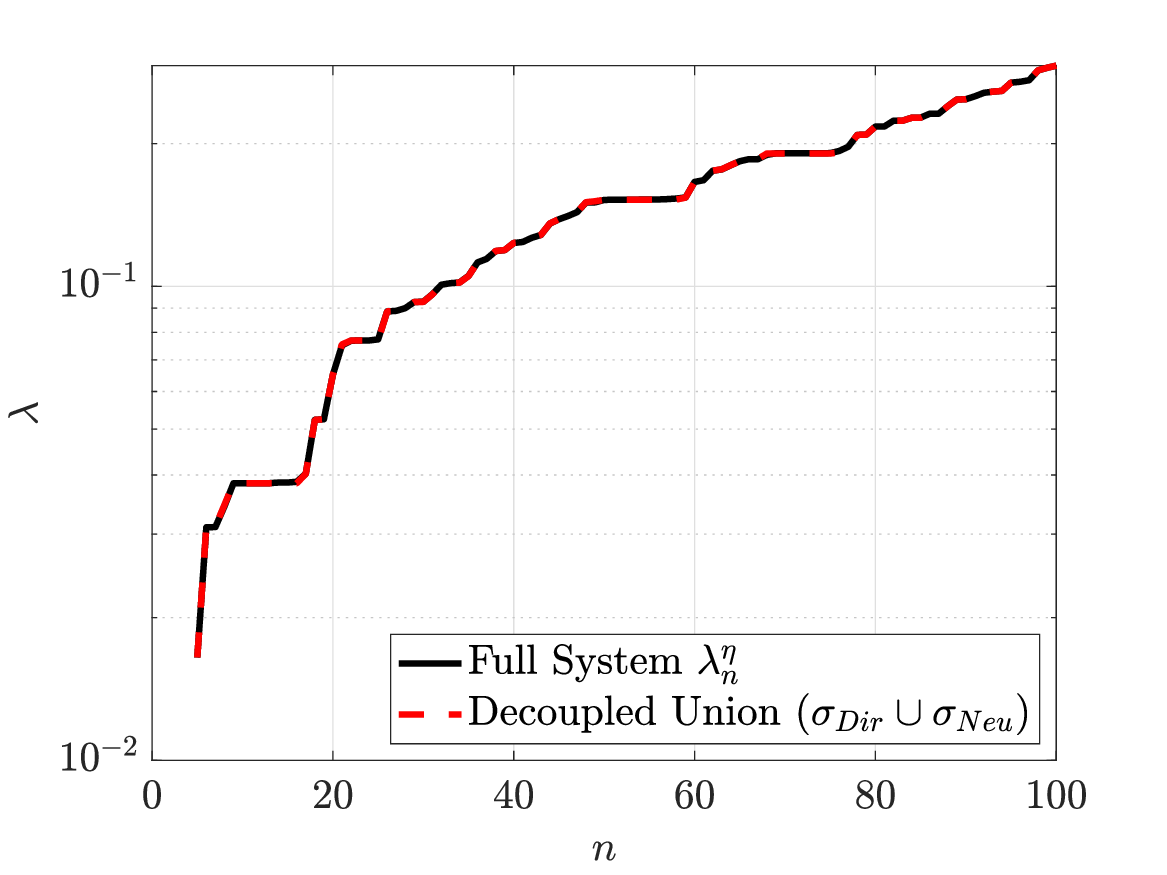}
        \caption{$\eta = 10^3$}
        \label{fig:sd3}
    \end{subfigure}
    \hfill
    \begin{subfigure}{0.48\textwidth}
        \includegraphics[width=\textwidth]{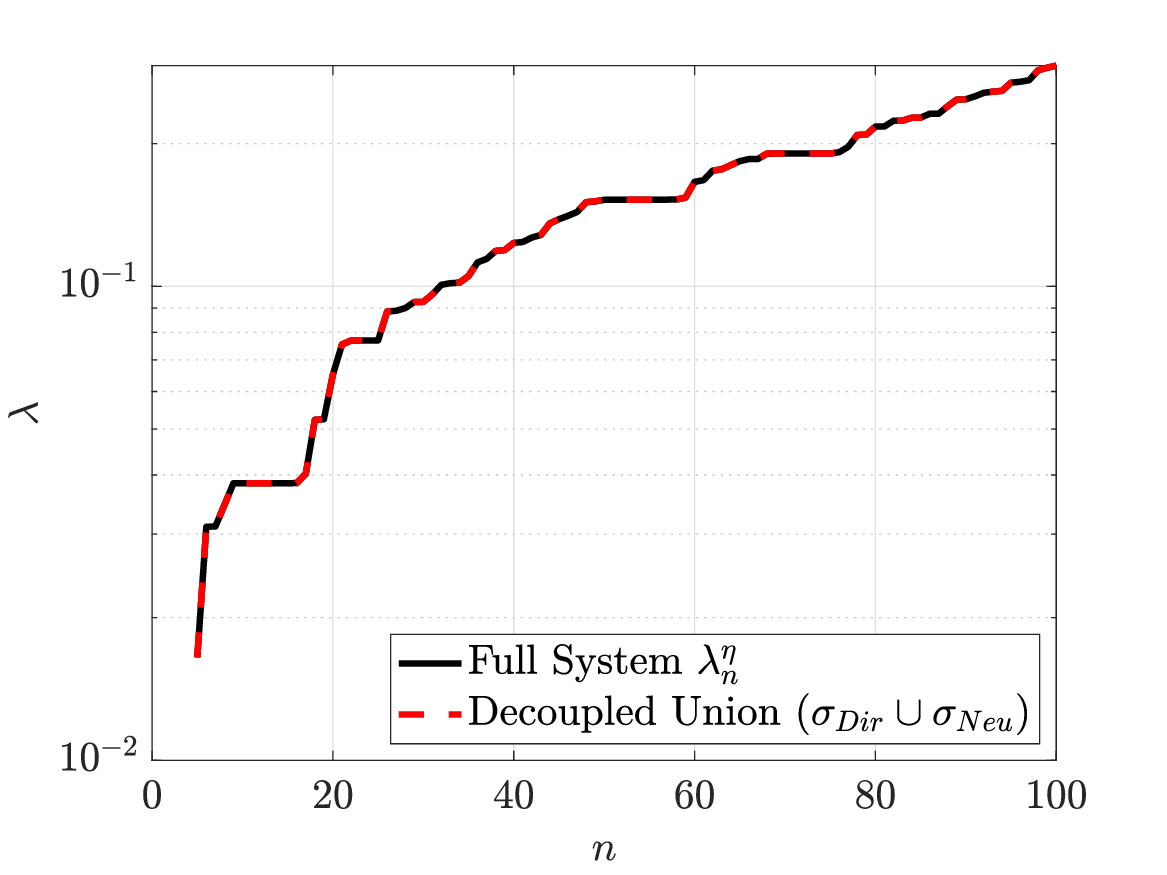}
        \caption{$\eta = 10^4$}
        \label{fig:sd4}
    \end{subfigure}
    \caption{Comparison of the eigenvalue sequence between the full multiscale system and the decoupled union spectrum for varying contrast parameters $\eta \in \{10^1, 10^2, 10^3, 10^4\}$. }
    \label{fig:spectral_decoupling}
\end{figure}

\cref{fig:spectral_decoupling} tracks the evolution of the eigenvalue sequence $\lambda_n^\eta$ of the full system for different values of $\eta$, comparing them with the theoretical decoupled union spectrum. 
For lower contrast values (e.g., $\eta = 10^1$), the multiscale eigenvalues exhibit noticeable deviations from the decoupled limit due to strong wave transmission across the interfaces. However, as the contrast systematically increases towards $\eta = 10^4$, the full system's eigenvalue curve uniformly converges to and tightly aligns with the theoretical decoupled spectrum. This dynamic sequence firmly confirms that extreme high-contrast completely cuts off the wave transmission across the interfaces, rendering the background matrix and internal inclusions mathematically independent.

\begin{figure}[htbp]
    \centering
    \begin{subfigure}{0.48\textwidth}
        \includegraphics[width=\textwidth]{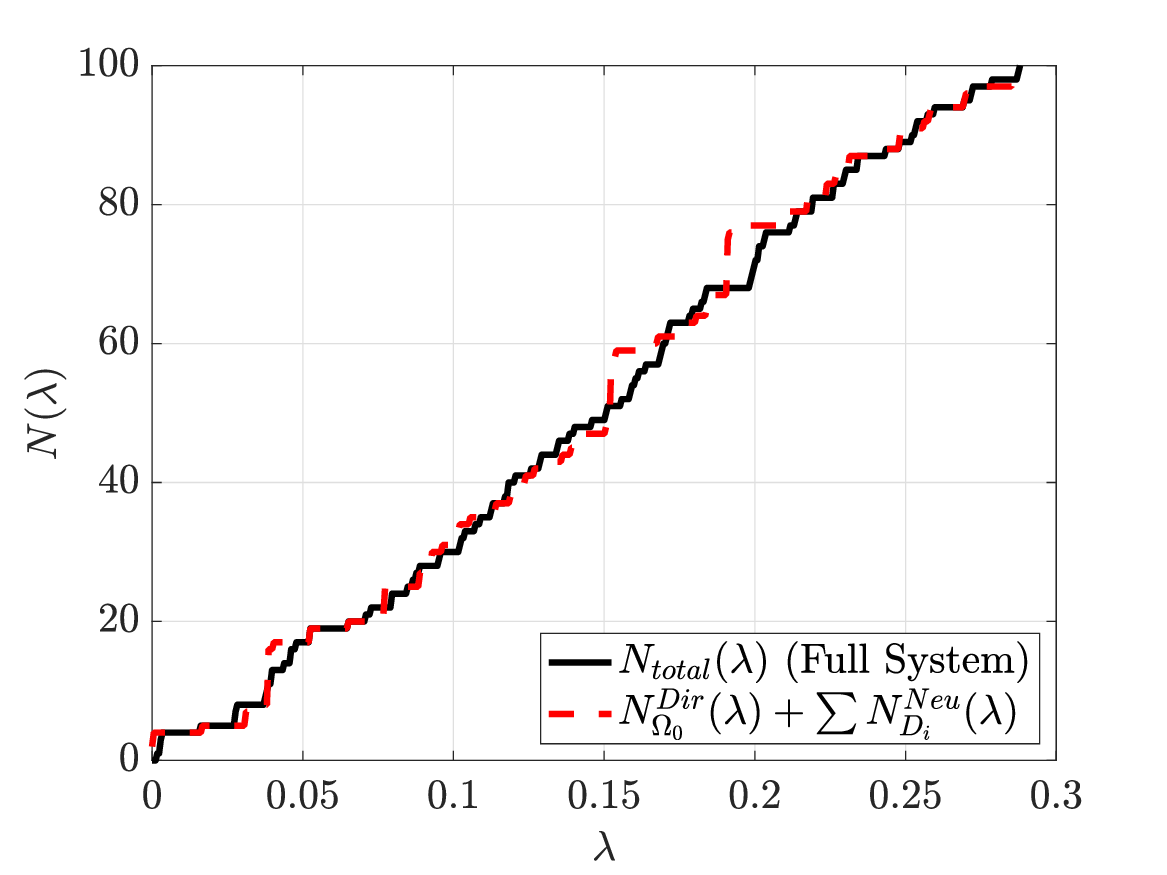}
        \caption{$\eta = 10^1$}
        \label{fig:gw1}
    \end{subfigure}
    \hfill
    \begin{subfigure}{0.48\textwidth}
        \includegraphics[width=\textwidth]{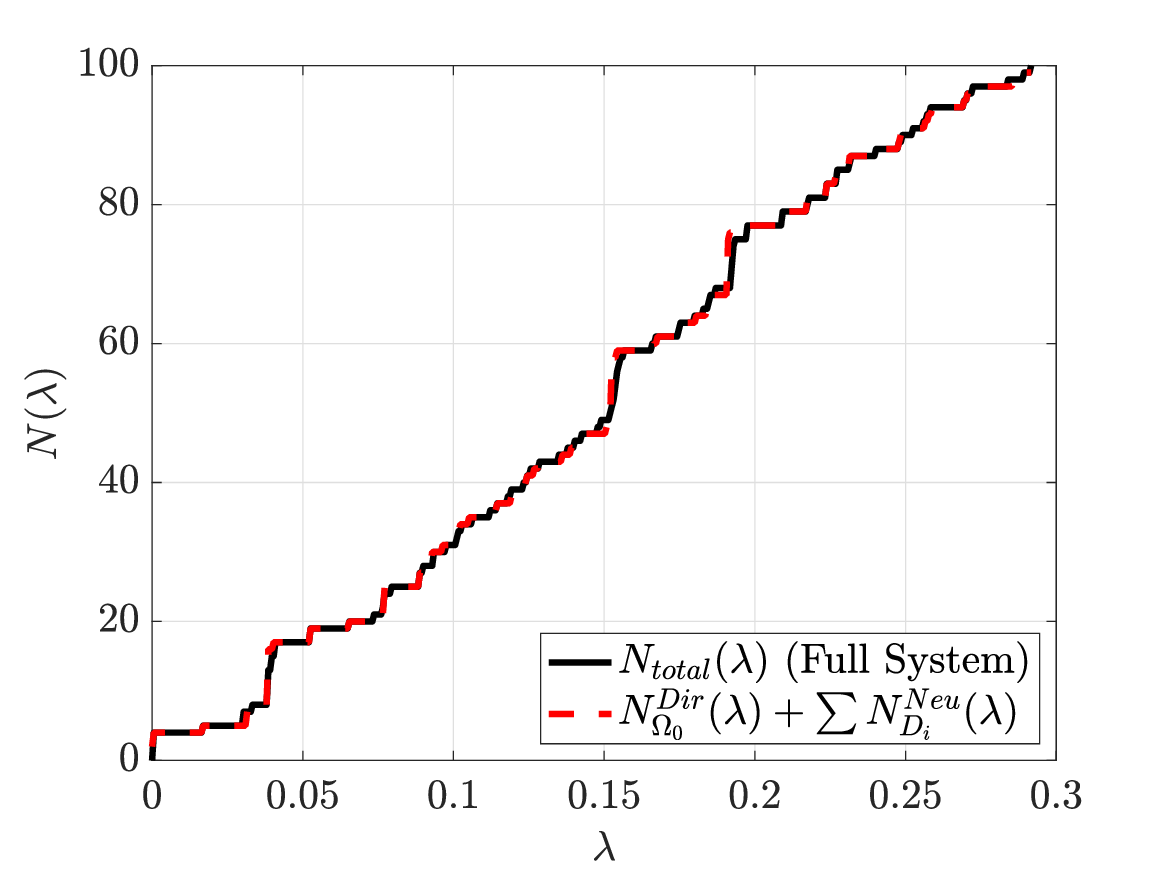}
        \caption{$\eta = 10^2$}
        \label{fig:gw2}
    \end{subfigure}
    
    \vspace{0.5cm}
    
    \begin{subfigure}{0.48\textwidth}
        \includegraphics[width=\textwidth]{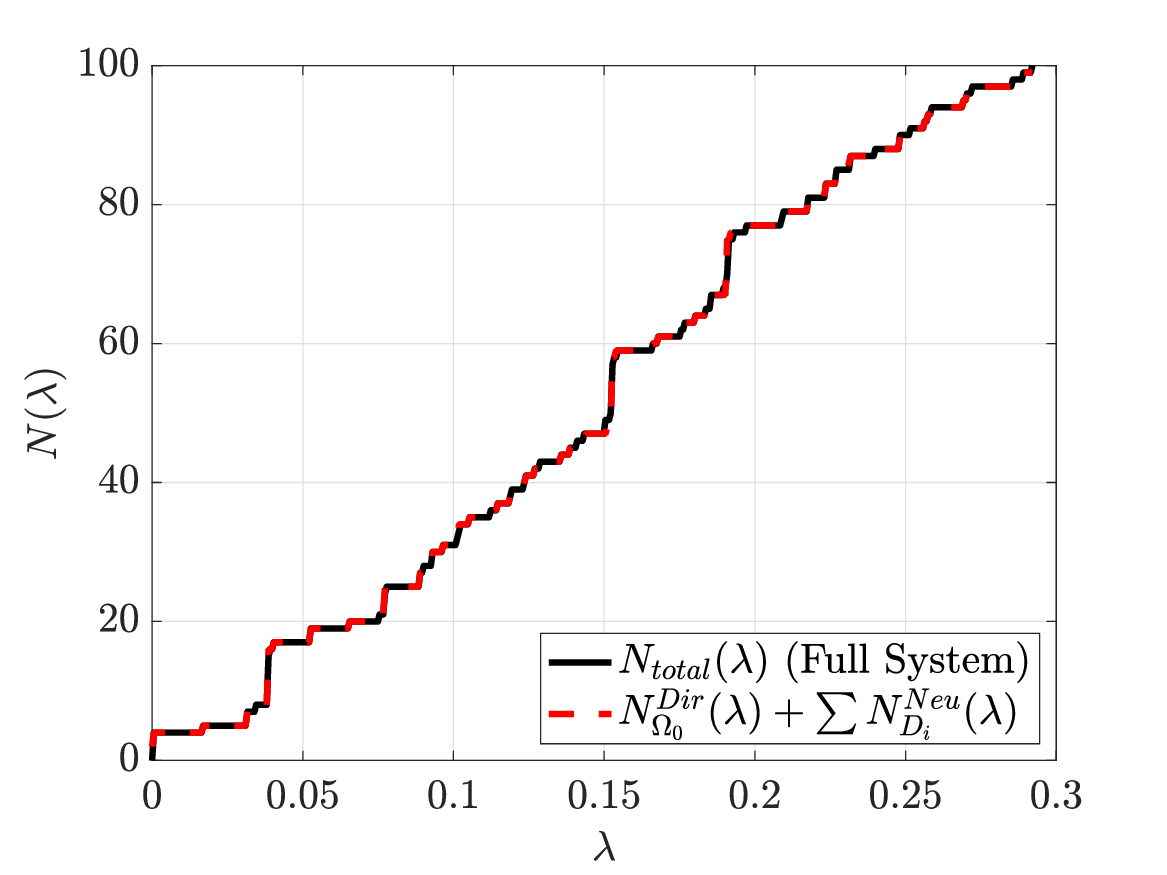}
        \caption{$\eta = 10^3$}
        \label{fig:gw3}
    \end{subfigure}
    \hfill
    \begin{subfigure}{0.48\textwidth}
        \includegraphics[width=\textwidth]{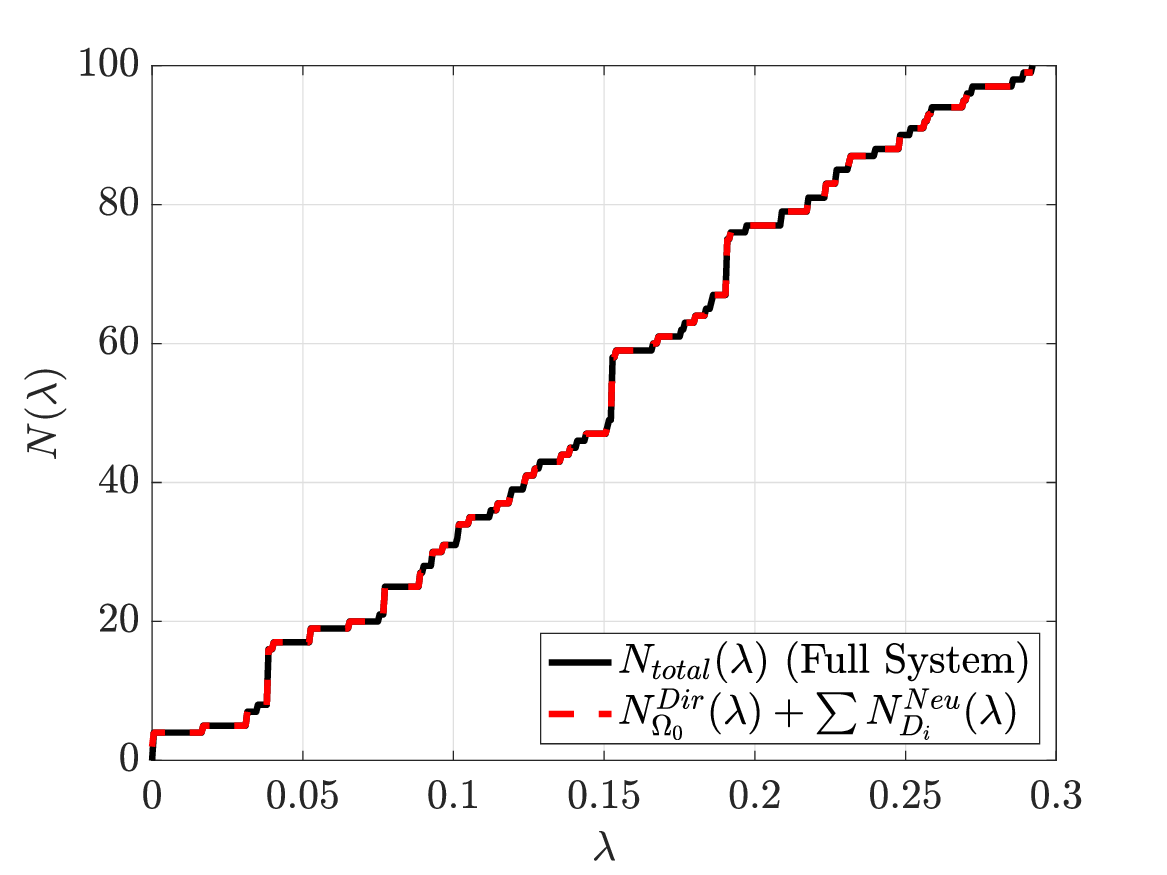}
        \caption{$\eta = 10^4$}
        \label{fig:gw4}
    \end{subfigure}
    \caption{Comparison of the eigenvalue counting functions $N(\lambda)$ demonstrating the asymptotic adherence to the generalized Weyl's law. }
    \label{fig:generalized_weyl}
\end{figure}

To further validate this spectral distribution macroscopically, we examine the eigenvalue counting function $N(\lambda) = \#\{n : \lambda_n \le \lambda\}$. \cref{fig:generalized_weyl} plots the counting function curves for the varying contrast parameters against the decoupled superposition $N_{\Omega_0}^{Dir}(\lambda) + \sum_{i=1}^M N_{D_i}^{Neu}(\lambda)$. The numerical results demonstrate a clear asymptotic convergence: the deviation between the full system's step-curve and the theoretical limit vanishes as $\eta$ grows. At $\eta = 10^4$, the curves virtually coincide across the entire spectral range. This rigorously verifies the spatial decoupling of high-frequency modes in high-contrast media and confirms that their linear superposition perfectly conforms to the multiscale Weyl's law.

\subsection{Asymptotic spectral gap and completeness}
In this subsection, we numerically verify the asymptotic spectral gap behavior and the spectral completeness with respect to the periodicity scale $\varepsilon$, as analyzed in \cref{sec:homogenization}. We consider the 2D scalar diffusion problem on a square macroscopic domain $\Omega = (0,1)^2$. The contrast parameter is $\eta = 10^5$. To observe the effect of vanishing periodicity, we vary the periodicity parameter $\varepsilon \in \{1/4, 1/8, 1/16, 1/32\}$.

\begin{figure}[htbp]
    \centering
    \begin{subfigure}{0.48\textwidth}
        \centering
        \includegraphics[width=\textwidth]{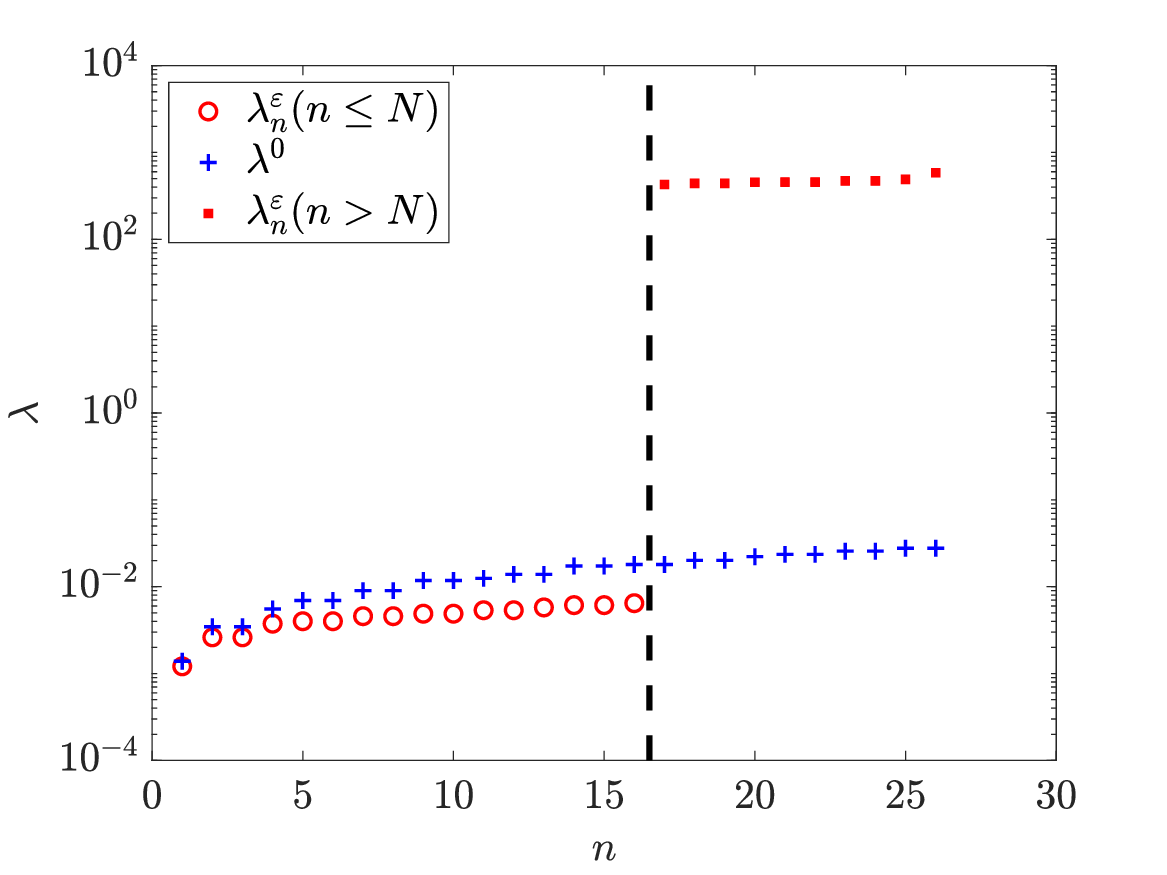}
        \caption{$\varepsilon = 1/4$ }
    \end{subfigure}
    \hfill
    \begin{subfigure}{0.48\textwidth}
        \centering
        \includegraphics[width=\textwidth]{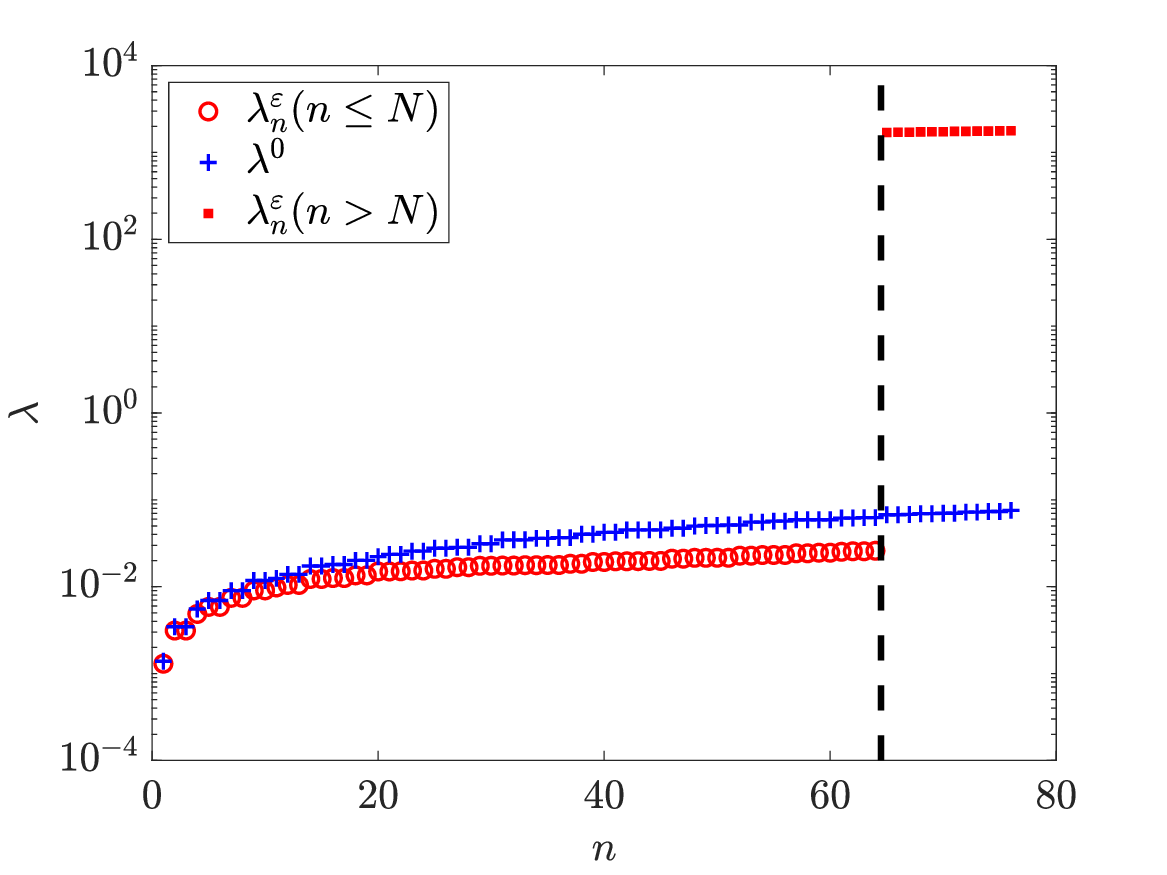}
        \caption{$\varepsilon = 1/8$ }
    \end{subfigure}
    \vspace{0.5cm}
    \begin{subfigure}{0.48\textwidth}
        \centering
        \includegraphics[width=\textwidth]{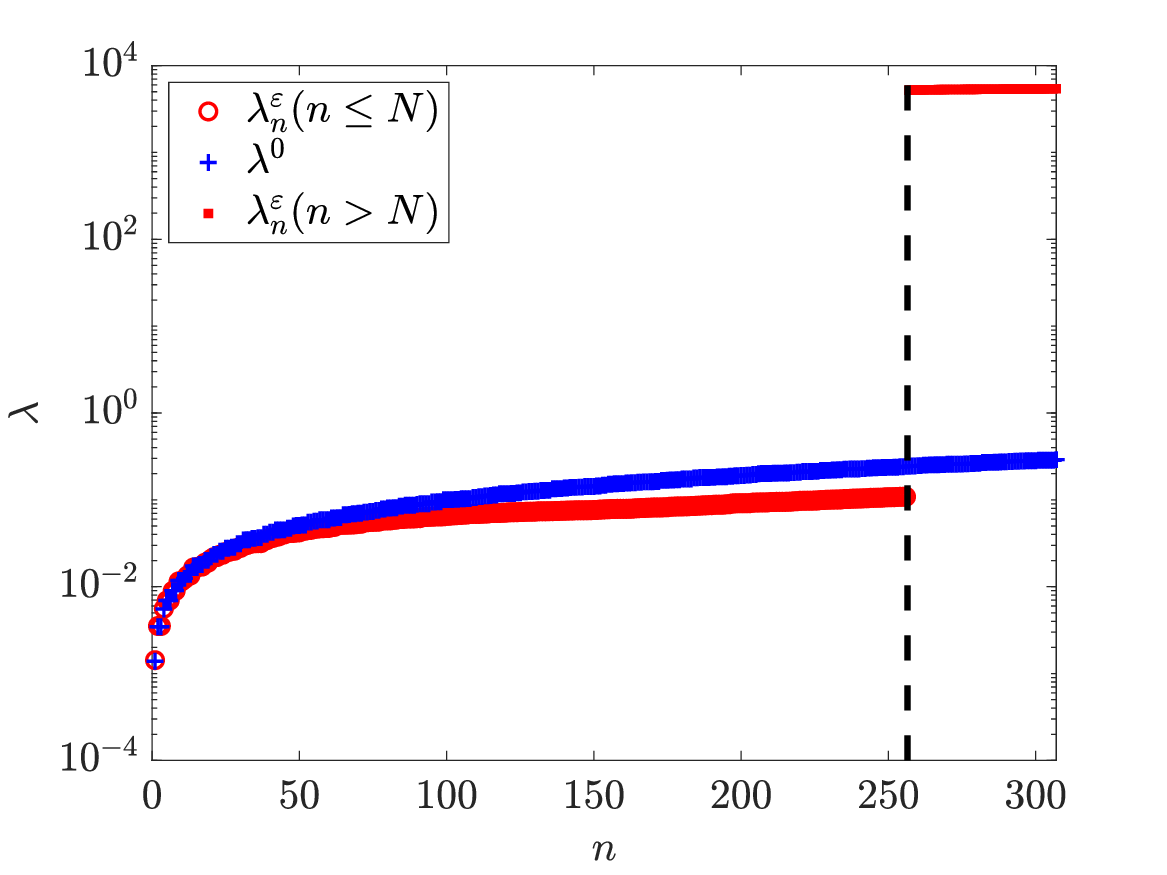}
        \caption{$\varepsilon = 1/16$ }
    \end{subfigure}
    \hfill
    \begin{subfigure}{0.48\textwidth}
        \centering
        \includegraphics[width=\textwidth]{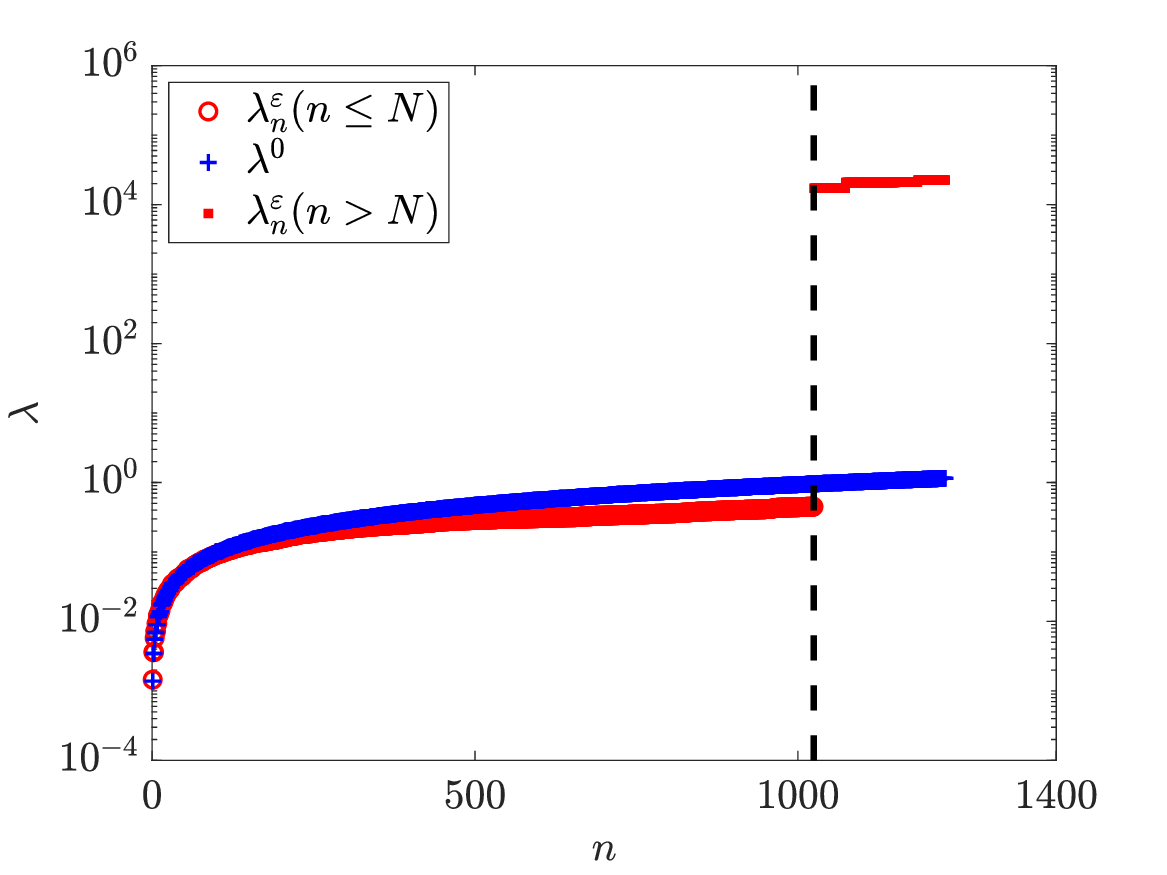}
        \caption{$\varepsilon = 1/32$ }
    \end{subfigure}
    \caption{Distribution of the eigenvalues for the multiscale operator with varying periodicity scales $\varepsilon \in \{1/4, 1/8, 1/16, 1/32\}$ and a fixed contrast $\eta = 10^5$.}
    \label{fig:varying_epsilon}
\end{figure}

\cref{fig:varying_epsilon} displays the eigenvalue distributions for the four different values of $\varepsilon$. The numerical results demonstrate the following key observations:

\begin{itemize}
    \item \textbf{Robust band separation and eigenvalue scaling:} For all tested values of $\varepsilon$, a spectral gap strictly emerges at $n = N(\varepsilon) + 1$. Both the low-frequency band and the high-frequency band shift upwards as $\varepsilon$ decreases. This behavior perfectly matches the $\varepsilon^{-2}$ gradient energy scaling, confirming that the relative spectral gap ratio remains invariant across all spatial scales.
    \item \textbf{Expansion of the topological eigenspace:} When the periodicity scale $\varepsilon$ decreases from $1/4$ to $1/32$, the number of modes $N(\varepsilon)$ located below the spectral gap increases quadratically, which agrees with the theoretical prediction $M(\varepsilon) \approx |\Omega|\varepsilon^{-2}$.
    \item \textbf{Spectral completeness:} The continuous rightward shift of the spectral gap in \cref{fig:varying_epsilon} provides direct numerical evidence for spectral completeness. For any fixed macroscopic mode index $n$ (e.g., $n=100$), there exists a critical scale $\varepsilon_0$ such that for all $\varepsilon < \varepsilon_0$, the $n$-th mode falls strictly within the low-frequency band $n \le N(\varepsilon)$. This visually confirms that the expanding topological subspace completely converges to the entire macroscopic spectrum as $\varepsilon \to 0$.
\end{itemize}

\section{Conclusions}
\label{sec:conclu}
In this paper, we presented a unified variational and topological framework to systematically characterize the spectral gaps and eigenvalue distributions of high-contrast multiscale PDEs. For standard operators with finite-dimensional kernels (e.g., scalar diffusion, linear elasticity, and fourth-order plates), we explicitly linked the spectral gap to the the dimension of the local null space and established strict lower bounds via classical coercivity inequalities. To address systems with infinite-dimensional kernels, such as Maxwell and grad-div operator, we employed De Rham complexes and a generalized Hodge decomposition to rigorously prove that the exact location of the universal topological spectral gap is completely determined by the aggregate Betti numbers of the inclusions. Furthermore, we established a spectral decoupling theorem based on a mass concentration dichotomy to formulate a new multiscale Weyl's law. Finally, we demonstrated that the expanding low-frequency topological eigenspace asymptotically and completely spans the entire macroscopic homogenized spectrum in the limit of vanishing periodicity.

Although this study provides a rigorous analysis of the asymptotic spectral behavior in high-contrast media, determining the explicit convergence rates of these eigenvalues remains a challenging open problem that will require the future development of error estimates. Another important direction for future research involves the quantitative study of the spectral measure. Exploring the geometric shape of the spectral measure, the detailed band-gap structures, and the topological and geometric dependencies of the spectral distribution will not only deepen the mathematical understanding of high-contrast PDE spectral theory, but also provide crucial theoretical guidance for designing next-generation multiscale model reduction techniques and robust domain decomposition preconditioners.

\bibliographystyle{siamplain}
\bibliography{references}
\end{document}